\documentclass{amsart}
\usepackage[margin=1.00in]{geometry}

\usepackage{amsfonts}

\usepackage{setspace}
\usepackage{graphicx}

\usepackage{hyperref}
\usepackage{graphicx}
\usepackage{amsmath, amsthm, amscd, amsfonts, amssymb, graphicx, color}
 \usepackage{url}
\usepackage{enumerate}
 \makeatletter
\let\reftagform@=\tagform@
\def\tagform@#1{\maketag@@@{(\ignorespaces\textcolor{blue}{#1}\unskip\@@italiccorr)}}
\renewcommand{\eqref}[1]{\textup{\reftagform@{\ref{#1}}}}
\makeatother
\usepackage{hyperref}
\hypersetup{colorlinks=true, linkcolor=red, anchorcolor=green,
citecolor=cyan, urlcolor=red, filecolor=magenta, pdftoolbar=true}

\usepackage{epsfig}        
\usepackage{epic,eepic}       

\newtheorem{theorem}{Theorem}
\theoremstyle{plain}

\newtheorem{corollary}{Corollary}

\newtheorem{definition}{Definition}
\newtheorem{example}{Example}

\newtheorem{proposition}{Proposition}
\newtheorem{remark}{Remark}

\numberwithin{equation}{section}

  \def\etal{et al.\,}

\begin{document}
\title[Some properties of $h$-${\rm{MN}}$-convexity]{Some properties of $h$-${\rm{MN}}$-convexity and Jensen's type inequalities}
\author[M.W. Alomari]{Mohammad W. Alomari}

\address {Department of Mathematics, Faculty of Science and Information	Technology, Irbid National University,  P.O. Box 2600, Irbid, P.C. 21110, Jordan.}
\email{mwomath@gmail.com}
\date{\today}


\begin{abstract}
 In this work, we introduce the class of
$h$-${\rm{MN}}$-convex functions by generalizing the concept of
${\rm{MN}}$-convexity and combining it with $h$-convexity. Namely,
Let $I,J$ be two intervals subset of $\left(0,\infty\right)$ such
that $\left(0,1\right)\subseteq J$ and $\left[a,b\right]\subseteq
I$. Consider  a non-negative function $h: (0,\infty)\to
\left(0,\infty\right)$ and let ${\rm{M}}:\left[0,1\right]\to
\left[a,b\right] $ $(0<a<b)$ be a Mean function given by
${\rm{{\rm{M}}}}\left(t\right)={\rm{{\rm{M}}}}\left( {h(t);a,b}
\right)$; where by ${\rm{{\rm{M}}}}\left( {h(t);a,b} \right)$ we
mean one of the following  functions: $A_h\left( {a,b}
\right):=h\left( {1 - t} \right)a + h(t) b$, $G_h\left( {a,b}
\right)=a^{h(1-t)} b^{h(t)}$ and $H_h\left( {a,b}
\right):=\frac{ab}{h(t) a + h\left( {1 - t} \right)b} =
\frac{1}{A_h\left( {\frac{1}{a},\frac{1}{b}} \right)}$; with the
property that ${\rm{{\rm{M}}}}\left( {h(0);a,b} \right)=a$ and
${\rm{M}}\left( {h(1);a,b} \right)=b$.

A function $f : I \to
\left(0,\infty\right)$ is said to be $h$-${\rm{{\rm{MN}}}}$-convex
(concave) if the inequality
\begin{align*}
f \left({\rm{M}}\left(t;x, y\right)\right) \le  (\ge) \,
{\rm{N}}\left(h(t);f (x), f (y)\right),
\end{align*}
holds  for  all $x,y \in I$ and $t\in [0,1]$, where M and N are two mean functions. In this way,  nine
classes of $h$-${\rm{MN}}$-convex functions are established  and
some of their analytic properties are
explored and investigated. Characterizations of each type are
given. Various Jensen's type inequalities and their converses are
proved.
\end{abstract}

\maketitle
\section{Introduction}
 Throughout this work, $I$ and $J$ are two
intervals subset of $\left(0,\infty\right)$ such that
$\left(0,1\right)\subseteq J$ and $\left[a,b\right]\subseteq I$
with $0<a<b$.     A function $f:I\to \mathbb{R}$ is called convex
iff
\begin{align}
f\left( {t\alpha +\left(1-t\right)\beta} \right)\le tf\left(
{\alpha} \right)+ \left(1-t\right) f\left( {\beta}
\right),\label{eq1.1}
\end{align}
for all points $\alpha,\beta \in I$ and all $t\in [0,1]$. If $-f$
is convex then we say that $f$ is concave. Moreover, if $f$ is
both convex and concave, then $f$ is said to be affine.

In 1978, Breckner \cite{B1} introduced the class of $s$-convex
functions (in the second sense), as follows:
\begin{definition}
\label{def1}  Let $I \subseteq \left[0,\infty\right)$ and $s\in
(0,1]$, a function $f : I\to \left[0,\infty\right)$ is
 $s$-convex function or that $f$ belongs to the
class $K^2_s\left(I\right)$ if for all $x,y \in I$ and $t \in
[0,1]$ we have
\begin{align*}
f \left(tx+\left(1-t\right)y\right) \le
t^sf\left(x\right)+\left(1-t\right)^sf\left(y\right).
\end{align*}
\end{definition}
In \cite{B2}, Breckner proved that every $s$-convex function
satisfies the H\"{o}lder condition of order $s$. Another proof of
this fact was given in \cite{P}. For more properties  regarding
$s$-convexity see \cite{B3} and \cite{HL}.

In 1985, E. K. Godnova and V. I. Levin (see \cite{GL} or
\cite{MPF}, pp. 410-433) introduced the following class of
functions:
\begin{definition}
\label{def2}   We say that $f : I \to \mathbb{R}$ is a
Godunova-Levin function or that $f$ belongs to the class $Q\left(I
\right)$ if for all $x,y \in I$ and $t \in (0,1)$ we have
\begin{align*}
f \left(tx+\left(1-t\right)y\right) \le
\frac{f\left(x\right)}{t}+\frac{f\left(y\right)}{1-t}.
\end{align*}
\end{definition}
In the same work, the authors proved that  all nonnegative
monotonic and nonnegative convex functions belong to this class.
For related works see \cite{DPP} and \cite{MP}.

In 1999, Pearce and Rubinov \cite{PR}, established a new type of
convex functions which is called $P$-functions.
\begin{definition}
\label{def3}  We say that  $f : I\to \mathbb{R}$ is $
 P $-function or that $f$ belongs to the class
$P\left(I \right)$ if for all $x,y \in I$ and $t \in [0,1]$ we
have
\begin{align*}
f \left(tx+\left(1-t\right)y\right) \le
 f\left(x\right)+ f\left(y\right).
\end{align*}
\end{definition}
Indeed, $Q(I) \supseteq P (I)$ and for applications it is
important to note that $P (I)$ also consists only of nonnegative
monotonic, convex and quasi-convex functions.  A related work was
considered in \cite{DPP} and \cite{TYD}.

In 2007, Varo\v{s}anec \cite{V} introduced the class of $h$-convex
functions which generalize convex, $s$-convex, Godunova-Levin
functions and $P$-functions. Namely, the $h$-convex function is
defined as a non-negative function $f : I \to \mathbb{R}$ which
satisfies
\begin{align*}
f\left( {t\alpha +\left(1-t\right)\beta} \right)\le
h\left(t\right) f\left( {\alpha} \right)+ h\left(1-t\right)
f\left( {\beta} \right),
\end{align*}
where $h$ is a non-negative function, $t\in \in (0, 1)\subseteq J$
and $x,y \in I $, where $I$ and $J$ are real intervals such that
$(0,1) \subseteq J $. Accordingly, some properties of $h$-convex
functions were discussed in the same work of Varo\v{s}anec. For
more results; generalization, counterparts and inequalities
regarding $h$-convexity see \cite{AOV},\cite{AAS},\cite{BV},
\cite{C}--\cite{D2},\cite{H},\cite{LW},\cite{M}, \cite{O},  \cite{SSY} and \cite{YBQ}

While he studying  $h$-convex functions, Alomari \cite{mwo} proposed a rational geometric and analytic meaning of $h$-convexity by introducing the concept of  $h$-cord as follows:
\begin{definition}\emph{(\cite{mwo})}
    The $h$-cord joining any two points
    $\left(x,f\left(x\right)\right)$ and
    $\left(y,f\left(y\right)\right)$ on the graph of $f$ is defined to
    be
    \begin{align}
    L\left(t;h\right):= \left[ {f\left( y \right) - f\left( x \right)}
    \right]h\left(\frac{t-x}{y-x}\right) +  f\left( x
    \right),\label{eq1.2}
    \end{align}
    for all $t\in [x,y]\subseteq I$ $(x<y)$.  In particular, if
    $h(t)=t$ then we obtain the well known form of chord, which is
    \begin{align*}
    L\left(t;t\right):=\frac{f\left( y \right) - f\left( x
        \right)}{y-x} \left( {t-x} \right)+ f\left( x \right).
    \end{align*}
\end{definition}
It's worth to mention that, if $h\left(0\right)=0$ and
$h\left(1\right)=1$, then $L\left(x;h\right)=  f\left( x \right)$
and $L\left(y;h\right)= f\left( y \right)$, so that the $h$-cord
$L$ agrees with $f$ at endpoints $x,y$, and this true for all such
$x,y \in
I$.

The $h$-convexity of a function $f : I \to \mathbb{R}$ means
geometrically that the points of the graph of $f$ are on or below
the $h$-chord joining the endpoints
$\left(x,f\left(x\right)\right)$ and
$\left(y,f\left(y\right)\right)$ for all $x,y \in I$, $x < y$. In
symbols, we write
\begin{align}
f\left(t\right)\le  \left[ {f\left( y \right) - f\left( x \right)}
\right]h\left(\frac{t-x}{y-x}\right) + f\left( x
\right)=L\left(t;h\right),\label{eq1.3}
\end{align}
for any $x \le t \le y$ and $x,y\in I$.
\begin{figure}[!h]
    \begin{center}
        \includegraphics[angle=0,width=2.6in]{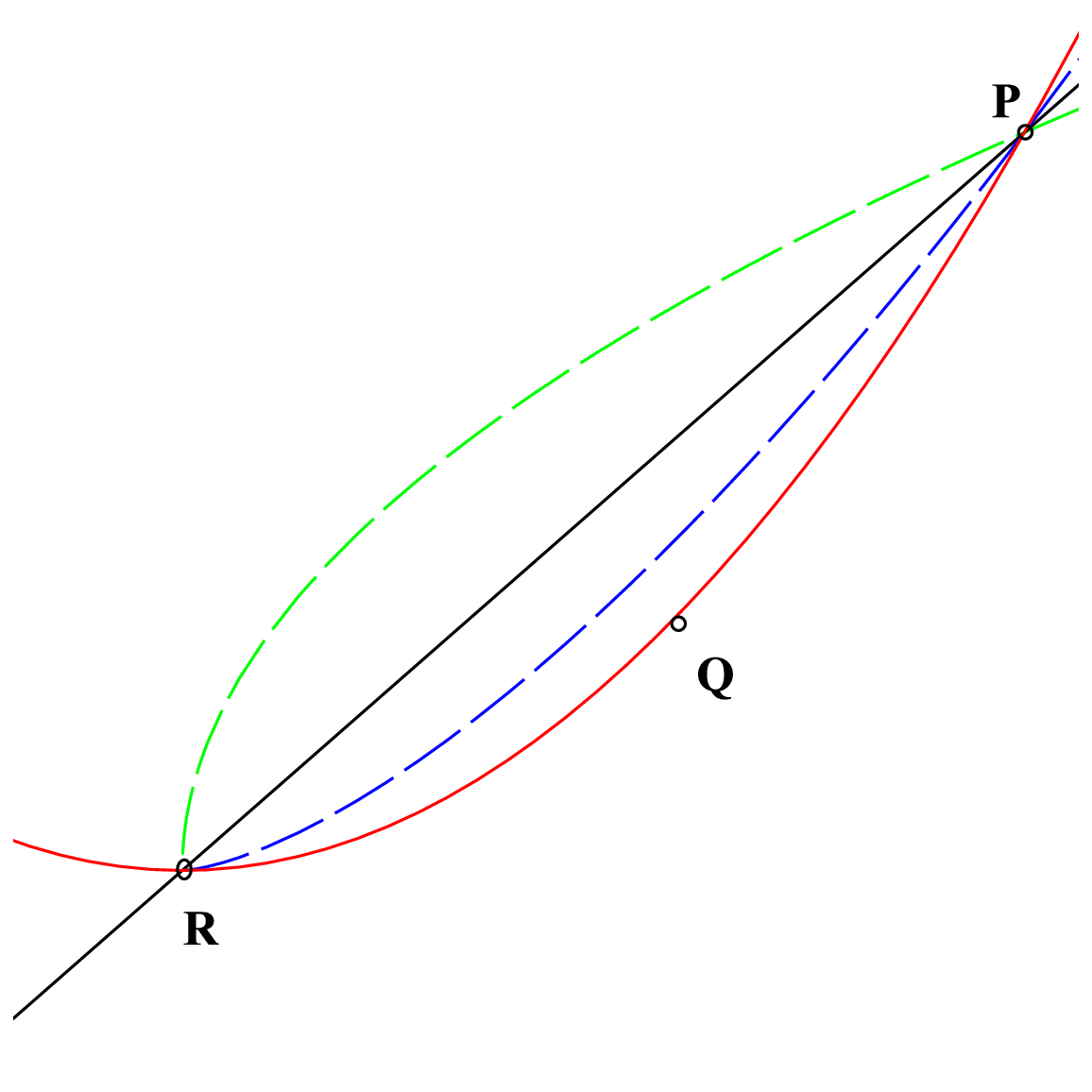}
    \end{center}
    \caption{The graph of $h_k(t)=t^k$, $k=\frac{1}{2}, 1 ,
        \frac{3}{2}$ (green, black, blue), respectively, and $f(t)=t^2$
        (red), $t\in [0,1]$.} \label{fig1}
\end{figure}
    Given any three non-collinear points $P,Q$ and $R$ on the graph of
$f$ with $Q$ between $P$ and $R$ (say $P<Q<R$). Let $h$ is
super(sub)multiplicative and $h\left(\alpha\right)\ge
(\le)\,\alpha$, for $\alpha \in \left(0,1\right)\subset J$. A
function $f$ is $h$--convex (concave) if $Q$ is on or below
(above) the $h$-chord $\widehat{PR}$ (see Figure \ref{fig1}).\\

\noindent \textbf{Caution:} In special case, for
$h_k\left(t\right) = t^k$, $t\in \left(0,1\right)$ the proposed geometric
interpretation is valid for $k \in (-1,0)\cup (0,\infty)$. In the
case that $k \le -1$ or $k=0$ the geometric meaning is
inconclusive so we exclude this case (and (and similar cases) from our proposal above.

\begin{definition}
	Let $h: J\to \left(0,\infty\right)$ be a non-negative
	function. Let $f : I \to \mathbb{R}$ be any function. We say $f$
	is $h$-midconvex ($h$-midconcave) if
	\begin{align*}
	f\left( {\frac{{x + y}}{2}} \right) \le (\ge)\,h\left(
	{\frac{1}{2}} \right)\left[ {f\left( x \right) + f\left( y
		\right)} \right]
	\end{align*}
	for  all $x,y \in I$.
\end{definition}
In particular, $f$ is locally $h$-midocnvex if and only if
\begin{align*}
h\left( {\frac{1}{2}} \right)\left[f\left( {x + p} \right) +
f\left( {x - p} \right)\right] - f\left( x \right) \ge 0,
\end{align*}
for  all $x\in \left(x-p,x+p\right)$, $p>0$.

Generalization of the well known Jensen convexity, could be stated as
follows:
\begin{theorem}
	Let $h: J\to \left(0,\infty\right)$ be a non-negative function
	such that $h\left(\alpha\right) \ge \alpha$, for all $\alpha \in
	(0,1)$. Let $f : I \to \mathbb{R}_+$ be a   nonnegative
	continuous function. $f$ is $h$-convex if and only if it is
	$h$-midconvex; i.e., the inequality
	\begin{align*}
	f\left( {\frac{{x + y}}{2}} \right) \le  h\left( {\frac{1}{2}}
	\right)\left[ {f\left( x \right) + f\left( y \right)} \right],
	\end{align*}
	holds for  all $x,y \in I$.
\end{theorem}

It's well knwon that every convex function is Lipschitz continuous. Moreover,   Breckner \cite{B2} (see also \cite{mwo1}, \cite{B3} and \cite{P}) proved that every $s$-convex functions is H\"{o}lder continuous of order $s\in (0,1]$. Recently, Alomari \cite{mwo} used the concept of control functions in numberical analysis to extend these facts in terms of  $h$-convex functions. Recall that a  function $h: J \subseteq [0,\infty) \to
[0,\infty]$ is called a control function if
\begin{enumerate}
    \item $h$ is nondecreasing,

    \item $\inf_{\delta > 0} h\left(\delta\right) = 0$.
\end{enumerate}

A function $f:I\to \mathbb{R}$ is $h$-continuous at $x_0$ if
$\left| f\left(x\right) - f\left(x_0\right)\right| \le h\left(|x-
x_0|\right)$, for all $x \in I$. Furthermore, a function is
continuous in   $x_0$ if it is $h$-continuous for some control
function $h$.

This approach leads us to refining the notion of continuity by
restricting the set of admissible control functions.

For a given set of control functions $\mathcal{C}$ a function is
$\mathcal{C}$-continuous if it is $h$-continuous for all $h \in
\mathcal{C}$. For example the  H\"{o}lder continuous functions of
order $\alpha \in (0,1]$  are defined by the set of control
functions
\begin{align*}
\mathcal{C}_{\alpha {\text{-H\"{o}lder}}}  =\left\{h |
h\left(\delta\right) = H \left|\delta\right|^{\alpha},  H >
0\right\}
\end{align*}
In case $\alpha=1$, the set $\mathcal{C}_{1- {\rm{Holder}}}$
contains all  functions satisfying the Lipschitz condition.

In \cite{mwo}, Alomari proved the following theorem.
\begin{theorem} \label{thm1}
    Let $(0,1)\subseteq J$, $h: J\to \left(0,\infty\right)$ be a
    control function which is supermultiplicative such that $h(\alpha)\ge \alpha$ for each
    $\alpha\in (0,1)$. Let $I$ be a real interval, $a,b\in \mathbb{R}$
    $(a<b)$ with $a,b$ in $ I^{\circ}$ (the interior of $I$). If $f :
    I \to \mathbb{R}$ is non-negative $h$-convex function on $[a,b]$,
    then $f$ is $h$-continuous on $[a,b]$.
\end{theorem}

We recall that, a function ${\rm{M}}:(0,\infty) \to (0,\infty)$ is
called a Mean function if
\begin{enumerate}
    \item  {\rm{Symmetry:}} ${\rm{M}}\left(x,y\right)={\rm{M}}\left(y,x\right)$.

    \item {\rm{Reflexivity:}} ${\rm{M}}\left(x,x\right)=x$.

    \item {\rm{Monotonicity:}} $\min\{x,y\} \le {\rm{M}}\left(x,y\right) \le \max\{x,y\}$.

    \item {\rm{Homogeneity:}} $ {\rm{M}}\left(\lambda x,\lambda y\right)=\lambda
    {\rm{M}}\left(x,y\right)$, for any positive scalar $\lambda$.
\end{enumerate}

The most famous and old known mathematical means are listed as
follows:
\begin{enumerate}
\item The arithmetic mean :
$$A := A\left( {\alpha ,\beta } \right) = \frac{{\alpha  + \beta
}}{2},\,\,\,\,\,\alpha ,\beta \in  \mathbb{R}_+.$$

\item The geometric mean :
$$G: = G\left( {\alpha ,\beta } \right) = \sqrt {\alpha \beta },\,\,\,\,\,\alpha ,\beta \in \mathbb{R}_+$$

\item The harmonic mean :
$$H: = H\left(
{\alpha ,\beta } \right) = \frac{2}{{\frac{1}{\alpha } +
\frac{1}{\beta }}},\,\,\,\,\,\alpha ,\beta \in \mathbb{R}_+
 - \left\{ 0 \right\}.$$
\end{enumerate}
In particular, we have the famous inequality $ H \le G \le A$.

In 2007, Anderson \etal in \cite{AVV} developed a systematic study
to the classical theory of continuous and midconvex functions, by
replacing a given mean instead of the arithmetic mean.
\begin{definition}
\label{def4}Let $f : I \to \left(0, \infty\right)$ be a
continuous function where $I \subseteq (0,\infty)$. Let ${\rm{M}}$
and ${\rm{N}}$ be any two Mean functions. We say $f$ is
${\rm{{\rm{MN}}}}$-convex (concave) if
\begin{align}
f \left({\rm{M}}\left( x, y\right)\right) \le  (\ge) \,
{\rm{N}}\left(f (x), f (y)\right), \label{eq1.5}
\end{align}
 for  all $x,y \in I$ and $t\in [0,1]$.
\end{definition}
In fact, the authors in \cite{AVV} discussed the midconvexity of
positive continuous real functions according to some Means. Hence,
the usual midconvexity is a special case when both mean values are
arithmetic means. Also, they studied the dependence of
${\rm{MN}}$-convexity on ${\rm{M}}$ and ${\rm{N}}$ and give
sufficient conditions for ${\rm{MN}}$-convexity of functions
defined by Maclaurin series.   For other works regarding
${\rm{MN}}$-convexity see \cite{N} and \cite{NP}.

The aim of this work, is to study the main properties of
$h$-${\rm{MN}}$-convex functions, such as; addition, product,
compositions and some functional type inequalities for some
classes. Jensen inequality and its consequences with their
converses play  significant roles in (almost) all areas of
Mathematics and Physics. For example, Jensen inequality used to
prove some important inequalities such as AM, GM, HM inequalities
and their consequences, moreover it can be used to generate some
more ramified inequalities. All this happens using the classical
concept of convex set and convex functions, but what happen when
we replace these terms by another convexity terms such as
$h$-${\rm{MN}}$-convexity?. In fact, discovering new Jensen type
inequalities will help us to  find, refine, and generate new
important inequalities e.g., AM-GM-HM type inequalities which have
a wide range of applications.

In this work, the class of $h$-${\rm{MN}}$-convex functions is
introduced. Generalizing and extending some classes of convex
functions are given. Some analytic properties for each class of
functions are explored and investigated. Characterizations of each
type of convexity are established. Some related Jensen's type
inequalities and their converses are proved.

\section{The $h$-${\rm{{\rm{MN}}}}$-convexity}

Let $h: J\to \left(0,\infty\right)$ be a non-negative
function. Define the
function ${\rm{M}}:\left[0,1\right]\to \left[a,b\right]$ given by
${\rm{{\rm{M}}}}\left(t\right)={\rm{{\rm{M}}}}\left( {h\left( {t} \right);a,b}
\right)$; where by ${\rm{{\rm{M}}}}\left( {t;a,b} \right)$ we mean
one of the following functions:
\begin{enumerate}
\item $A_h\left( {a,b} \right):=h\left( {t} \right)a +  h\left( {1-t} \right)b
 ;\qquad \text{The generalized Arithmetic
Mean}$.\\

\item $G_h\left( {a,b}
\right)=a^{h\left( {t} \right)} b^{h\left( {1-t} \right)};\qquad\qquad\,\,\,\,\,\,\, \text{The generalized Geometric Mean}$.\\

\item $H_h\left( {a,b} \right):=\frac{ab}{h\left( {1-t} \right) a + h\left( {  t}
\right)b} = \frac{1}{A_h\left( {\frac{1}{a},\frac{1}{b}} \right)};
\qquad\,\,\,\,\, \text{The generalized
Harmonic Mean}$.\\
\end{enumerate}
Note that ${\rm{{\rm{M}}}}\left( {h\left( {0} \right);a,b} \right)=a$ and
${\rm{M}}\left( {h\left( {1} \right);a,b} \right)=b$. Clearly, for $h(t)=t$ with $t=\frac{1}{2}$,
the means $A_{\frac{1}{2}}$, $G_{\frac{1}{2}}$ and
$H_{\frac{1}{2}}$, respectively; represents the midpoint of the
$A_{t}$, $G_{t}$ and $H_{t}$, respectively; which was discussed in
\cite{AVV} in viewing of Definition \ref{def4}.

For $h(t)=t$, we note that the above means are related with celebrated
AM-GM-HM inequality
\begin{align*}
H_t\left( {a,b} \right)\le G_t\left( {a,b} \right) \le  A_t\left(
{a,b} \right),\qquad \forall \,\, t\in [0,1].
\end{align*}
Indeed, one can easily prove  more general form of the above inequality; that is if $h$ is positive increasing on $[0,1]$ then the generalized AM-GM-HM inequality
is given by
\begin{align}
\label{AM-GM-HM}H_h\left( {a,b} \right)\le G_h\left( {a,b} \right) \le  A_h\left(
{a,b} \right),\qquad \forall \,\, t\in [0,1]\,\,\,\,\text{and}\,\,\,\,  a,b>0.
\end{align}

\subsection{Basic properties of $h$-${\rm{{\rm{MN}}}}$-convex  functions}

The Definition \ref{def4} can be extended according to the defined
mean ${\rm{{\rm{M}}}}\left( {t;a,b} \right)$, as follows:
 Let $f : I \to \left(0, \infty\right)$
be any function. Let ${\rm{M}}$ and ${\rm{N}}$ be any two Mean
functions. We say $f$ is ${\rm{{\rm{MN}}}}$-convex (concave) if
\begin{align*}
f \left({\rm{M}}\left(t;x, y\right)\right) \le  (\ge) \,
{\rm{N}}\left(t;f (x), f (y)\right),
\end{align*}
 for  all $x,y \in I$ and $t\in [0,1]$.

Next, we introduce the class of ${\rm{M_tN_h}}$-convex
functions by generalizing the concept of ${\rm{M_tN_t}}$-convexity
and combining it with $h$-convexity.
\begin{definition}
\label{def5} Let $h: J\to \left(0,\infty\right)$ be a non-negative
function. Let $f : I \to \left(0, \infty\right)$ be any function.
Let ${\rm{M}}:\left[0,1\right]\to \left[a,b\right]$ and
${\rm{N}}:\left(0,\infty\right)\to \left(0,\infty\right)$ be any
two Mean functions. We say $f$ is $h$-${\rm{{\rm{MN}}}}$-convex
(-concave) or that $f$ belongs to the class
$\overline{\mathcal{MN}}\left(h,I\right)$
($\underline{\mathcal{MN}}\left(h,I\right)$) if
\begin{align}
f \left({\rm{M}}\left(t;x, y\right)\right) \le  (\ge) \,
{\rm{N}}\left(h(t);f (x), f (y)\right),\label{eq2.4}
\end{align}
 for  all $x,y \in I$ and $t\in [0,1]$.
\end{definition}
Clearly, if ${\rm{M}}\left(t;x, y\right)=A_t\left( {x,y}
\right)={\rm{N}}\left(t;x, y\right)$, then Definition \ref{def5}
reduces to the original concept of $h$-convexity. Also, if we
assume $f$ is continuous,  $h(t)=t$ and $t=\frac{1}{2}$ in
\eqref{eq2.4}, then the Definition \ref{def5} reduces to the
Definition \ref{def4}.

The cases of $h$-${\rm{{\rm{MN}}}}$-convexity are given with
respect to a certain mean, as follow:
\begin{enumerate}
\item $f$ is  ${\rm{A_tG_h}}$-convex iff
\begin{align}
 f\left( {t\alpha  + \left( {1 -
t} \right)\beta } \right) \le \left[ {f\left( \alpha  \right)}
\right]^{h\left( t \right)} \left[ {f\left( \beta \right)}
\right]^{h\left(1- t \right)}, \qquad  0\le t\le 1,\label{eqhAG}
\end{align}

\item $f$ is  ${\rm{A_tH_h}}$-convex iff
\begin{align}
 f\left( {t\alpha  + \left( {1 -
t} \right)\beta } \right) \le\frac{{f\left( \alpha \right)f\left(
\beta  \right)}}{{h\left( 1-t \right)f\left( \alpha  \right) +
h\left( { t} \right)f\left( \beta  \right)}}, \qquad  0\le t\le
1.\label{eqhAH}
\end{align}

\item $f$ is  ${\rm{G_tA_h}}$-convex iff
\begin{align}
f\left( {\alpha ^t \beta ^{1 - t} } \right) \le h\left( {t}
\right)f\left( \alpha \right) + h\left( {1 - t} \right)f\left(
\beta  \right), \qquad 0\le t\le 1.\label{eqhGA}
\end{align}

\item $f$ is ${\rm{G_tG_h}}$-convex iff
\begin{align}
f\left( {\alpha ^t \beta ^{1 - t} } \right) \le \left[ {f\left(
\alpha  \right)} \right]^{h\left( {t} \right)} \left[ {f\left(
\beta  \right)} \right]^{h\left( {1 - t} \right)}, \qquad  0\le
t\le 1.\label{eqhGG}
\end{align}

\item $f$ is ${\rm{G_tH_h}}$-convex iff
\begin{align}
f\left( {\alpha ^t \beta ^{1 - t} } \right) \le \frac{{f\left(
\alpha  \right)f\left( \beta  \right)}}{{h\left( {1- t}
\right)f\left( \alpha  \right) + h\left( { t} \right)f\left( \beta
\right)}}, \qquad  0\le t\le 1.\label{eqhGH}
 \end{align}

\item $f$ is  ${\rm{H_tA_h}}$-convex iff
\begin{align}
f\left( {\frac{{\alpha \beta }}{{t\alpha  + \left( {1 - t}
\right)\beta }}} \right) \le h\left( {1-t} \right)f\left( \alpha
\right) + h\left( { t} \right)f\left( \beta  \right), \qquad 0\le
t\le 1.\label{eqhHA}
 \end{align}
 \item $f$ is  ${\rm{H_tG_h}}$-convex iff
\begin{align}
f\left( {\frac{{\alpha \beta }}{{t\alpha  + \left( {1 - t}
\right)\beta }}} \right) \le \left[ {f\left( \alpha  \right)}
\right]^{h\left( { 1-t} \right)} \left[ {f\left( \beta  \right)}
\right]^{h\left( {  t} \right)}, \qquad 0\le t\le 1.\label{eqhHG}
 \end{align}

\item $f$ is ${\rm{H_tH_h}}$-convex iff
\begin{align}
f\left( {\frac{{\alpha \beta }}{{t\alpha  + \left( {1 - t}
\right)\beta }}} \right) \le \frac{{f\left( \alpha  \right)f\left(
\beta  \right)}}{{h\left( {t} \right)f\left( \alpha  \right) +
h\left( {1 - t} \right)f\left( \beta  \right)}}, \qquad  0\le t\le
1.\label{eqhHH}
\end{align}
\end{enumerate}

\begin{remark}
In all previous cases,  $h(t)$ and $h(1-t)$ are not equal to zero
at the same time. Therefore, if
 $h(0)=0$ and $h(1)=1$, then the Mean function ${\rm{N}}$ satisfying
 the conditions ${\rm{N}}\left( {h\left( 0 \right),f\left( x \right),f\left( y \right)}
\right) = f\left( x \right)$ and $ {\rm{N}}\left( {h\left( 1
\right),f\left( x \right),f\left( y \right)} \right) = f\left(y
\right) $.
\end{remark}
\begin{remark}
According to the Definition \ref{def5},  we may extend the classes
$Q(I), P(I)$
 and $K_s^2$ by replacing the  arithmetic mean by another given
 one.
 Let ${\rm{M}}:\left[0,1\right]\to \left[a,b\right]$
and ${\rm{N}}:\left(0,\infty\right)\to \left(0,\infty\right)$ be
any two Mean functions.
\begin{enumerate}
\item Let $s\in (0,1]$, a function $f : I\to
\left(0,\infty\right)$ is ${\rm{M_tN_{t^s}}}$-convex function or
that $f$ belongs to the class
$K^2_s\left(I;{\rm{M_t}},{\rm{N_{t^s}}}\right)$ if for all $x,y \in I$
and $t \in [0,1]$ we have
\begin{align}
f \left({\rm{M}}\left(t;x, y\right)\right) \le {\rm{N}}\left(t^s;f
(x), f (y)\right).\label{eq2.10}
\end{align}

\item  We say that $f : I \to \left(0, \infty\right)$ is an
extended Godunova-Levin function or that $f$ belongs to the class
$Q\left(I;{\rm{M_t}},{\rm{N_{1/t}}}\right)$ if for all $x,y \in I$ and
$t \in (0,1)$ we have
\begin{align}
f \left({\rm{M}}\left(t;x, y\right)\right) \le
{\rm{N}}\left(\frac{1}{t};f (x), f (y)\right).\label{eq2.11}
\end{align}

\item We say that  $f : I\to \left(0,\infty\right)$ is
$P$-${\rm{M_tN_{t=1}}}$-function or that $f$ belongs to the class
$P\left(I;{\rm{M_t}},{\rm{N_1}}\right)$ if for all $x,y \in I$ and
$t \in [0,1]$ we have
\begin{align}
f \left({\rm{M}}\left(t;x, y\right)\right) \le {\rm{N}}\left(1;f
(x), f (y)\right).\label{eq2.12}
\end{align}
 In \eqref{eq2.10}--\eqref{eq2.12}, setting ${\rm{M}}\left(t;x,
y\right)= {\rm{A_t}}\left(x, y\right)={\rm{N}}\left(t;x,
y\right)$, we then refer to the original definitions of these
class of convexities (see Definitions \ref{def1}--\ref{def3}).
\end{enumerate}
\end{remark}

\begin{remark}\cite{V}
\label{remark2}Let $h$ be a non-negative function such that
$h\left(t\right) \ge t$ for $t\in \left(0,1\right)$. For instance
$h_r\left(t\right) = t^r$, $t\in \left(0,1\right)$ has that
property. In particular, for $r\le 1$, if $f$ is a non-negative
${\rm{M_tN_t}}$-convex function on $I$, then for $x,y\in I$, $t
\in (0,1)$ we have
\begin{align*}
f \left({\rm{M}}\left(t;x, y\right)\right) \le {\rm{N}}\left(t;f
(x), f (y)\right) \le {\rm{N}}\left(t^r;f (x), f (y)\right)=
{\rm{N}}\left(h\left(t\right);f (x), f (y)\right),
\end{align*}
for all $r\le 1$ and $t\in \left(0,1\right)$. So that $f$ is
 ${\rm{M_tN_h}}$-convex. Similarly, if the function satisfies
the property $h\left(t\right) \le t$ for $t \in \left(0,1\right)$,
then $f$ is a non-negative  ${\rm{M_tN_h}}$-concave. In
particular, for $r\ge 1$, the function $h_r(t)$ has that property
for $t\in \left(0,1\right)$.
 So that if $f$ is a non-negative ${\rm{M_tN_t}}$-concave
function on $I$, then for $x,y\in I$, $t \in (0,1)$ we have
\begin{align*}
f \left({\rm{M}}\left(t;x, y\right)\right) \ge {\rm{N}}\left(t;f
(x), f (y)\right) \ge {\rm{N}}\left(t^r;f (x), f
(y)\right)={\rm{N}}\left(h\left(t\right);f (x), f (y)\right),
\end{align*}
for all $r\ge 1$ and $t\in \left(0,1\right)$, which means that $f$
is ${\rm{M_tN_h}}$-concave.
\end{remark}

\begin{remark}
There exists an $h$-${\rm{MN}}$-convex function which is
${\rm{MN}}$-convex.  As  shown by Varo\v{s}anec (see Examples 6
and 7 in  \cite{V}), one can generate $h$-${\rm{MN}}$-convex
functions  but not ${\rm{MN}}$-convex.
\end{remark}

Next, we give an extended  generalization of Theorem 2.4 in
\cite{AVV}. This simply can help to illustrate the concept of
$h$-${\rm{MN}}$-convex functions.
\begin{theorem}\label{thm1}
Let $h: J\to \left(0,\infty\right)$ be a positive function. $f:I
\to \left(0, \infty\right)$ be any function. In parts (4)--(9),
let $I = (0, \tau)$, $0<\tau< \infty$.
\begin{enumerate}
\item $f$ is  ${\rm{A_tA_h}}$-convex (-concave) if and only if\, $f$\,
is $h$-convex ($h$-concave).

\item $f$ is  ${\rm { A_tG_h}}$-convex (-concave) if and only if \,$\log
f$ \,is $h$-convex (-concave).

\item $f$ is  ${\rm{A_tH_h}}$-convex (-concave) if and only if  \,
$\frac{1}{f(x)}$  is $h$-concave (-convex).

\item $f$ is ${\rm{ G_tA_h}}$-convex (-concave) on $I$ if and only
if\, $f\left(\tau {\rm{e}}^{-t }\right)$ is $h$-convex (-concave).

\item $f$ is  ${\rm{G_tG_h}}$-convex (-concave) if and only if \,$\log
f\left(\tau {\rm{e}}^{-t }\right)$  is $h$-convex (-concave) on
$\left(0,\infty\right)$.

\item $f$ is ${\rm{G_tH_h}}$-convex (-concave) if and only if
\,$\frac{1}{f\left(\tau {\rm{e}}^{-t }\right)}$ is $h$-concave
(-convex) on $\left(0,\infty\right)$.

\item $f$ is ${\rm{H_tA_h}}$-convex (-concave) if and only if
\,$f\left({\textstyle{1 \over x }}\right)$ is  $h$-convex
(-concave) on $\left({\textstyle{1 \over \tau }},\infty\right)$.

 \item $f$ is ${\rm{H_tG_h}}$-convex (-concave) if and only if\,
 $\log f\left({\textstyle{1 \over x }}\right)$ is  $h$-convex (-concave) on
 $\left({\textstyle{1 \over \tau }},\infty\right)$.

\item $f$ is ${\rm{H_tH_h}}$-convex (-concave) if and only if\,
$\frac{1}{f\left({\textstyle{1 \over x }}\right)}$ is $h$-concave
(-convex) on $\left({\textstyle{1 \over \tau }},\infty\right)$.
\end{enumerate}
\end{theorem}

\begin{proof}

\begin{enumerate}
\item Follows by definition.

\item  Employing \eqref{eqhAG} in the Definition \ref{def5}, we
have
\begin{align*}
&f \left(A_t\left( {a,b} \right)\right) \le  (\ge) \,
G\left(h(t);f (a), f (b)\right)
\\
&\Leftrightarrow f \left(\left( {1 - t} \right)a + tb\right) \le
(\ge) \,\left[f\left( {a} \right)\right]^{h\left(1-t\right)}
\left[f\left( {b} \right)\right]^{h\left(t\right)}
\\
&\Leftrightarrow \log f \left(\left( {1 - t} \right)a + tb\right)
\le (\ge) \,h\left(1-t\right)\log \left[f\left(a\right)\right] +
h\left(t\right)\log \left[f\left(b\right)\right],
\end{align*}
which proves the result.

\item Employing \eqref{eqhAH} in the Definition \ref{def5}, we
have
\begin{align*}
&f \left(A_t\left( {a,b} \right)\right) \le  (\ge) \,
H\left(h(t);f (a), f (b)\right)
\\
&\Leftrightarrow f \left(\left( {1 - t} \right)a + tb\right) \le
(\ge) \frac{{f\left( a \right)f\left( b \right)}}{{h\left( t
\right)f\left( a  \right) + h\left( {1 - t} \right)f\left(b
\right)}}
\\
&\Leftrightarrow   \frac{1}{f \left(\left( {1 - t} \right)a +
tb\right)} \ge (\le) \, \frac{h\left(1-t\right)}{f\left(a\right)}
+ \frac{h\left(t\right)}{f\left(b\right)},
\end{align*}
which proves the result.

\item Employing \eqref{eqhGA} in the Definition \ref{def5} and
substituting $a = \tau e^{ - r}$ and $b = \tau e^{ - s} $, we have
\begin{align*}
&f \left(G_t\left( {a,b} \right)\right) \le  (\ge) \,
A\left(h(t);f (a), f (b)\right)
\\
&\Leftrightarrow f \left(a^{1 - t} b^t\right) \le (\ge) h\left(
{1-t} \right)f\left( a\right) + h\left( { t} \right)f\left( b
\right)
\\
&\Leftrightarrow f \left(\tau e^{ - \left[ {r\left( {1 - t}
\right) + st} \right]}
 \right) \le (\ge) h\left( {1-t} \right)f\left( \tau e^{ - r}\right) +
h\left( { t} \right)f\left(  \tau e^{ - s} \right),
\end{align*}
which proves the result.

\item Employing \eqref{eqhGG} in the Definition \ref{def5} and
 substituting $a = \tau e^{ - r}$ and $b = \tau e^{ - s} $, we
have
\begin{align*}
&f \left(G_t\left( {a,b} \right)\right) \le  (\ge) \,
G\left(h(t);f (a), f (b)\right)
\\
&\Leftrightarrow f \left(a^{1 - t} b^t\right) \le
(\ge)\left[f\left( a\right)\right]^{ h\left( {1-t} \right)}
\left[f\left( b \right)\right]^{ h\left( {t} \right)}
\\
&\Leftrightarrow \log f \left(\tau e^{ - \left[ {r\left( {1 - t}
\right) + st} \right]}
 \right) \le (\ge) h\left( {1-t} \right)\log f\left( \tau e^{ - r}\right) +
h\left( { t} \right)\log  f\left(  \tau e^{ - s} \right),
\end{align*}

\item Employing \eqref{eqhGH} in the Definition \ref{def5}  and
 substituting $a = \tau e^{ - r}$ and $b = \tau e^{ - s} $, we
have, we have
\begin{align*}
&f \left(G_t\left( {a,b} \right)\right)  \le  (\ge) \,
H\left(h(t);f (a), f (b)\right)
\\
&\Leftrightarrow f \left(a^{1 - t} b^t\right) \le (\ge)
\frac{{f\left( a \right)f\left( b \right)}}{{h\left( t
\right)f\left( a  \right) + h\left( {1 - t} \right)f\left(b
\right)}}
\\
&\Leftrightarrow   \frac{1}{f \left(a^{1 - t} b^t\right)} \ge
(\le) \, \frac{h\left(1-t\right)}{f\left(a\right)} +
\frac{h\left(t\right)}{f\left(b\right)}
\\
&\Leftrightarrow  \frac{1}{f \left(\tau e^{ - \left[ {r\left( {1 -
t} \right) + st} \right]}\right)} \ge (\le) \,
\frac{h\left(1-t\right)}{f\left(\tau e^{ - r} \right)} +
\frac{h\left(t\right)}{f\left(\tau e^{ - s} \right)},
\end{align*}
which proves the result.

\item  Let $g(x)=f\left(\frac{1}{x}\right)$ and let $a,b \in
\left( {\textstyle{1 \over \tau}} ,\infty\right)$ with $a<b$, so
that $a,b\in \left(0,\tau\right)$. Then $f$ is  ${\rm{H_tA_h}}$-convex
(-concave) on $\left(0,\tau\right)$ if and only if
\begin{align*}
&f \left(\frac{1}{H_t\left( {a,b} \right)}\right) \le  (\ge) \,
A\left(h(t);\frac{1}{f\left(a\right)},
\frac{1}{f\left(b\right)}\right)
\\
&\Leftrightarrow f\left( {\frac{1}{\textstyle{ab \over ta+(1-t)b
}}} \right) \le \left(  \ge \right)h\left( t \right)f\left(
{\frac{1}{b}} \right) + h\left( {1 - t} \right)f\left(
{\frac{1}{a}} \right)
\\
&\Leftrightarrow  g\left( {\frac{{ab}}{{ta + \left( {1 - t}
\right)b}}} \right) \le \left(  \ge  \right)h\left( 1-t
\right)g\left( a \right) + h\left( {t} \right)g\left( b \right),
\end{align*}
which proves the result.

\item  Let $g(x)=\log f\left(\frac{1}{x}\right)$ and let $a,b \in
\left( {\textstyle{1 \over \tau }} ,\infty\right)$ with $a<b$, so
that $a,b\in \left(0, \tau\right)$. Then $f$ is
 ${\rm{H_tG_h}}$-convex (-concave) on $\left(0,\tau\right)$ if and
only if
\begin{align*}
&f \left(\frac{1}{H_t\left( {a,b} \right)}\right) \le  (\ge) \,
G\left(h(t);f (a), f (b)\right)
\\
&\Leftrightarrow   f\left( {\frac{t}{b} + \frac{{\left( {1 - t}
\right)}}{a}} \right) \le \left(  \ge  \right) \left[f\left(
{\frac{1}{b}} \right) \right]^{h\left( t \right)}  \left[f\left(
{\frac{1}{a}} \right) \right]^{h\left( {1 - t} \right)}
\\
&\Leftrightarrow  \log f\left( {\frac{t}{b} + \frac{{\left( {1 -
t} \right)}}{a}} \right) \le \left(  \ge  \right) h\left( t
\right) \log f\left( {\frac{1}{b}} \right) + h\left( {1 - t}
\right)\log f\left( {\frac{1}{a}} \right)
\\
&\Leftrightarrow  g\left( {\frac{{ab}}{{ta + \left( {1 - t}
\right)b}}} \right) \le \left(  \ge  \right)h\left( t
\right)g\left( b \right) + h\left( {1 - t} \right)g\left( a
\right),
\end{align*}
which proves the result.

\item Let $g(x)= \frac{1}{f\left(\frac{1}{x}\right)}$ and let $a,b
\in \left( {\textstyle{1 \over \tau }} ,\infty\right)$ with $a<b$,
so that $a,b\in \left(0,\tau\right)$. Then $f$ is
 ${\rm{H_tH_h}}$-convex (-concave) on $\left(0,\tau\right)$ if and
only if
\begin{align*}
&f \left(\frac{1}{H_t\left( {a,b} \right)}\right) \le  (\ge) \,
H\left(h(t);f (a), f (b)\right)
\\
&\Leftrightarrow   f\left( {\frac{t}{b} + \frac{{\left( {1 - t}
\right)}}{a}} \right) \le \left(  \ge  \right) \frac{{f\left(
\frac{1}{a} \right)f\left( \frac{1}{b} \right)}}{{h(1-t)f\left(
\frac{1}{b} \right) + h\left( { t} \right)f\left( \frac{1}{a}
\right)}}
\\
&\Leftrightarrow  \frac{1}{f\left( {\frac{t}{b} + \frac{{\left( {1
- t} \right)}}{a}} \right) } \ge \left(  \le  \right)
\frac{{h(1-t)f\left( \frac{1}{b} \right) + h\left( {t}
\right)f\left( \frac{1}{a} \right)}}{f\left( \frac{1}{a}
\right)f\left( \frac{1}{b} \right)}
\\
&\Leftrightarrow  \frac{1}{f\left(  {  \frac{{ta + \left( {1 - t}
\right)b}}{{ab}} } \right) } \ge \left(  \le \right) \frac{{h(1-t)
}}{f\left( \frac{1}{a} \right)}+ \frac{{  h\left( {t} \right)
}}{f\left( \frac{1}{b} \right)}
\\
&\Leftrightarrow  g\left( {\frac{{ab}}{{ta + \left( {1 - t}
\right)b}}} \right) \ge \left(  \le  \right) h\left( {1 - t}
\right)g\left( a \right)+h\left( t \right)g\left( b \right),
\end{align*}
which proves the result.
\end{enumerate}
\end{proof}
\begin{proposition}
Let $h: J\to \left(0,\infty\right)$ be a non-negative   function.
Then
\begin{align*}
\begin{array}{*{20}c}
   { f \,\,\text{is    ${\rm{A_tH_h}}$-convex}  } &  \Longrightarrow  & {f \,\,\text{is    ${\rm{A_tG_h}}$-convex} } &  \Longrightarrow  & {f \,\,\text{is    ${\rm{A_tA_h}}$-convex} }  \\
    \Downarrow f \nearrow  &   &  \Downarrow f \nearrow   &   &  \Downarrow f \nearrow   \\
   {f \,\,\text{is  ${\rm{G_tH_h}} $-convex} } &  \Longrightarrow  & {f \,\,\text{is    ${\rm{G_tG_h}}$-convex} } &  \Longrightarrow  & {f \,\,\text{is    ${\rm{G_tA_h}}$-convex} }  \\
    \Downarrow f \nearrow  &   &  \Downarrow f \nearrow  &   &  \Downarrow f \nearrow   \\
    { f \,\,\text{is   ${\rm{H_tH_h}} $-convex}   } &  \Longrightarrow  & {f \,\,\text{is   ${\rm{H_tG_h}}$-convex}  } &  \Longrightarrow  & {f \,\,\text{is   ${\rm{H_tA_h}}$-convex} }.  \\
\end{array}
\end{align*}
By $f \nearrow$  we mean that $f$ is increasing and by $f \searrow$  we mean that $f$ is decreasing. For $h$-concavity and
decreasing monotonicity, the implications are reversed.
\end{proposition}
\begin{proof}
The proof of each statement follows from Definition \ref{def5} and
by noting that  for an increasing function $h$ we have $H_h\left( {a,b} \right)\le G_h\left( {a,b}
\right) \le A_h\left( {a,b} \right)$, for all $t\in [0,1]$.
Furthermore, and for instance we note that if  $f$ is
 ${\rm{A_tH_h}}$-convex, therefore we have
\begin{align*}
f\left( {{\rm{A}}_\alpha  \left( {x,y} \right)} \right)  = f\left(
{\alpha x + \left( {1 - \alpha } \right)y} \right) &\le
\frac{{f\left( x \right)f\left( y \right)}}{{h\left( {1 - \alpha }
\right)f\left( x \right) + h\left( \alpha  \right)f\left( y
\right)}}
\\
&= \frac{1}{{\frac{{h\left( {1 - \alpha } \right)}}{{f\left( y
\right)}} + \frac{{h\left( \alpha \right)}}{{f\left( x \right)}}}}
\\
&= {\rm H}\left( {h\left( \alpha  \right),f\left( x
\right),f\left( y \right)} \right),
\end{align*}
which is employing for  $g(t)=\frac{1}{f(t)}$, i.e.,
\begin{align*}
g\left( {{\rm{A}}_\alpha  \left( {x,y} \right)} \right) = g\left(
{\alpha x + \left( {1 - \alpha } \right)y} \right) =
\frac{1}{{f\left( {\alpha x + \left( {1 - \alpha } \right)y}
\right)}} &\ge \frac{{h\left( {1 - \alpha } \right)}}{{f\left( y
\right)}} + \frac{{h\left( \alpha  \right)}}{{f\left( x \right)}}
\\
&= h\left( {1 - \alpha } \right)g\left( y \right) + h\left( \alpha
\right)g\left( x \right)
\\
&={\rm A}\left( {h\left( \alpha  \right),g\left( x \right),g\left(
y \right)} \right),
\end{align*}
and this shows that $g$ is  ${\rm{A_tA_h}}$-concave.
\end{proof}

Thus, one can see the implications in Theorem \ref{thm1} are strict,
as shown by the following example:
\begin{example}
    Let $h$ be a non-negative function such that $h\left(t\right)\ge
    t$ for all $t \in \left(0,1\right)$. In particular, let
    $h\left(t\right)=h_k\left(t\right)=t^k$, $k\le 1$ and $t\in
    (0,1)$. The functions
    \begin{enumerate}
        \item $f\left(x\right)=\cosh \left(x\right)$ is
         ${\rm{A_tG_h}}$-convex, hence
         ${\rm{G_tG_h}}$-convex and  ${\rm{H_tG_h}}$-convex,
        on $(0, \infty)$. But it is not  ${\rm{A_tH_h}}$--convex,
        nor  ${\rm{G_tH_h}}$--convex, nor
         ${\rm{H_tH_h}}$--convex.

        \item $f\left(x\right)=\arcsin \left(x\right)$ is
         ${\rm{A_tA_h}}$-convex but it is ${\rm{A_tG_h}}$-concave
        for all $0\le x \le 1$.

        \item  $f\left(x\right)={\rm{e}}^x $ is
         ${\rm{G_tG_h}}$-convex and  ${\rm{H_tG_h}}$-convex,
        but neither  ${\rm{G_tH_h}}$-convex nor
         ${\rm{H_tH_h}}$-convex, for all $x>0$.

        \item $f\left(x\right)=\log\left(1+x\right)$ is
         ${\rm{G_tA_h}}$-convex  but
         ${\rm{G_tG_h}}$-concave for all $0<x<1$.

        \item $f\left(x\right)={\rm{e}}^{-x}$ is
     ${\rm{H_tA_h}}$-convex for $k \le \frac{1}{2}$ but not
         ${\rm{H_tG_h}}$-convex for all $0<x<1$. Also, $f$ is
         ${\rm{H_tA_t}}$-convex  but not  ${\rm{H_tG_t}}$-convex for
        all $x>1$.
    \end{enumerate}
\end{example}

\begin{proposition}
Let $h_1,h_2: J\to \left(0,\infty\right)$ be two positive positive
function with the property that $h_2(t)\le h_1(t)$ for all $t\in
(0,1)$. If $f$ is  ${\rm{{\rm{M_tN_{h_2}}}}}$-convex then
 ${\rm{{\rm{M_tN_{h_1}}}}}$-convex and if $f$ is
 ${\rm{{\rm{M_tN_{h_1}}}}}$-concave  then
${\rm{{\rm{M_tN_{h_2}}}}}$-concave.
\end{proposition}

\begin{proof}
From Definition \ref{def5} we have
\begin{align*}
f \left({\rm{M}}\left(t;x, y\right)\right) \le (\ge)\,
{\rm{N}}\left(h_2(t);f (x), f (y)\right) \le (\ge)\,
{\rm{N}}\left(h_1(t);f (x), f (y)\right),
\end{align*}
which is required.
\end{proof}

\begin{proposition}
If $f$ and $g$ are two  ${\rm{{\rm{M_tN_h}}}}$-convex and $\lambda
>0$, then $f+g$, $\lambda f$ and $\max\{f,g\}$.
\end{proposition}

\begin{proof}
The proof follows by  Definition \ref{def5}.
\end{proof}

 \begin{proposition}
Let $f$ and $g$ be a similarly  ordered functions. If $f$ is
 ${\rm{{\rm{A_tA_{h_1}}}}}$-convex
\emph{(${\rm{{\rm{G_tA_{h_1}}}}}$-convex,
 ${\rm{{\rm{H_tA_{h_1}}}}}$-convex)}, $g$ is
 ${\rm{{\rm{A_tA_{h_2}}}}}$-convex
\emph{(${\rm{{\rm{G_tA_{h_2}}}}}$-convex,
 ${\rm{{\rm{H_tA_{h_2}}}}}$-convex)}, respectively; and
$h\left(t\right)+h\left(1-t\right) \le c$, where
$h\left(t\right):=max\{h_1\left(t\right),h_2\left(t\right)\}$ and
$c$ is a fixed positive real number. Then the product $(fg)$ is
${\rm{{\rm{A_tA_{c\cdot h}}}}}$-convex
\emph{(${\rm{{\rm{G_tA_{c\cdot h}}}}}$-convex,
 ${\rm{{\rm{H_tA_{c\cdot h}}}}}$-convex)}, respectively.
\end{proposition}

\begin{proof}
Since $f$ and $g$ are similarly ordered functions we have
\begin{align*}
f\left( x \right)g\left( x \right)+f\left( y \right)g\left( y
\right) \ge f\left( x \right)g\left( y \right) +g\left( x
\right)f\left( y \right).
\end{align*}
Let $t$ and $s$ be positive numbers such that $t+s = 1$. Then we
obtain
\begin{align*}
 &\left( {fg} \right)\left( {A_t \left( {x,y} \right)} \right) \\
  &= \left( {fg} \right)\left( {sx + ty} \right) \\
  &\le  \left[ {h_1 \left( s \right)f\left( x \right) + h_1 \left( t \right)f\left( y \right)} \right]\left[ {h_2 \left( s \right)g\left( x \right) + h_2 \left( t \right)g\left( y \right)} \right] \\
  &\le h^2 \left( s \right)f\left( x \right)g\left( x \right) + h\left( t \right)h\left( s \right)\left[ {f\left( y \right)g\left( x \right) + f\left( x \right)g\left( y \right)} \right] + h^2 \left( t \right)f\left( y \right)g\left( y \right) \\
  &\le  h^2 \left( s \right)f\left( x \right)g\left( x \right) + h\left( t \right)h\left( s \right)\left[ {f\left( x \right)g\left(x \right) + f\left( y \right)g\left( y \right)} \right] + h^2 \left( t \right)f\left( y \right)g\left( y \right) \\
  &= \left( {h\left( s \right) + h\left( t \right)} \right)\left( {h\left( s \right)\left( {fg} \right)\left( x \right) + h\left( t \right)\left( {fg} \right)\left( y \right)}
  \right)\\
  &= c\cdot h\left( s \right)\left( {fg} \right)\left( x \right) + c\cdot h\left( t \right)\left( {fg} \right)\left( y \right)\\
  &=A\left( {c \cdot h(t);\left( {fg} \right)\left( x \right),\left( {fg}
\right)\left( y \right)} \right),
\end{align*}
which shows that $(fg)$ is ${\rm{{\rm{A_tA_{c\cdot h}}}}}$-convex. The cases when $fg$ is  ${\rm{{\rm{G_tA_{c\cdot h}}}}}$-convex or  ${\rm{{\rm{H_tA_{c\cdot h}}}}}$-convex, are follow in similar manner.
\end{proof}

 \begin{corollary}\label{cor2}
Let $f$ and $g$ be an oppositely  ordered functions. If $f$ is
${\rm{{\rm{A_tA_{h_1}}}}}$-concave
\emph{(${\rm{{\rm{G_tA_{h_1}}}}}$-concave,
 ${\rm{{\rm{H_tA_{h_1}}}}}$-concave)}, $g$ is
 ${\rm{{\rm{A_tA_{h_2}}}}}$-concave
\emph{(${\rm{{\rm{G_tA_{h_2}}}}}$-concave,
 ${\rm{{\rm{H_tA_{h_2}}}}}$-concave)}, respectively; and
$h\left(t\right)+h\left(1-t\right) \ge c$, where
$h\left(t\right):=min\{h_1\left(t\right),h_2\left(t\right)\}$ and
$c$ is a fixed positive real number. Then the product $(fg)$ is
$(c\cdot h)$-${\rm{{\rm{A_tA_h}}}}$-concave
\emph{(${\rm{{\rm{G_tA_h}}}}$-concave,
${\rm{{\rm{H_tA_h}}}}$-concave)}, respectively.
\end{corollary}

\begin{proposition}
If $f$ is  ${\rm{{\rm{A_tG_{h_1}}}}}$-convex
\emph{(${\rm{{\rm{G_tG_{h_1}}}}}$-convex,
 ${\rm{{\rm{H_tG_{h_1}}}}}$-convex)} and $g$ is
 ${\rm{{\rm{A_tG_{h_2}}}}}$-convex
\emph{(${\rm{{\rm{G_tG_{h_2}}}}}$-convex,
 ${\rm{{\rm{H_tG_{h_2}}}}}$-convex)}, respectively;   and $h\left(t\right):=max\{h_1\left(t\right),h_2\left(t\right)\}$, where
$h\left(t\right)+h\left(1-t\right) \le c$
. Then the
product $(fg)$ is  ${\rm{{\rm{A_tG_{c\cdot h}}}}}$-convex
\emph{(${\rm{{\rm{G_tG_{c\cdot h}}}}}$-convex,
 ${\rm{{\rm{H_tG_{c\cdot h}}}}}$-convex)}, respectively.
\end{proposition}

\begin{proof}
let $t \in (0,1)\subseteq J$, then
\begin{align*}
 &\left( {fg} \right)\left( {A_t \left( {x,y} \right)} \right) \\
 &= \left( {fg} \right)\left( {\left( {1 - t} \right)x + ty} \right) \\
  &\le \left\{ {\left[ {f\left( x \right)} \right]^{h_1 \left( {1 - t} \right)} \left[ {f\left( y \right)} \right]^{h_1 \left( t \right)} } \right\} \cdot \left\{ {\left[ {g\left( x \right)} \right]^{h_2 \left( {1 - t} \right)} \left[ {g\left( y \right)} \right]^{h_2 \left( t \right)} } \right\} \\
  &= \left[ {f\left( x \right)} \right]^{h_1 \left( {1 - t} \right)} \left[ {g\left( x \right)} \right]^{h_2 \left( {1 - t} \right)}  \cdot \left[ {f\left( y \right)} \right]^{h_1 \left( t \right)} \left[ {g\left( y \right)} \right]^{h_2 \left( t \right)}  \\
  &\le \left[ {\left( {fg} \right)\left( x \right)} \right]^{h\left( {1 - t} \right)}  \cdot \left[ {\left( {fg} \right)\left( y \right)} \right]^{h \left( t \right)}  \\
  &= G\left( {h\left( t \right),\left( {fg} \right)\left( x \right),\left( {fg} \right)\left( y \right)}
  \right),
\end{align*}
which shows that $(fg)$ is  ${\rm{{\rm{A_tG_h}}}}$-convex.  The
cases when $fg$ is ${\rm{{\rm{G_tG_{c\cdot h}}}}}$-convex or
 ${\rm{{\rm{H_tG_{c\cdot h}}}}}$-convex, are follow in similar
manner.
\end{proof}

\begin{corollary}\label{cor3}
If $f$ is  ${\rm{{\rm{A_tG_{h_1}}}}}$-concave
\emph{(${\rm{{\rm{G_tG_{h_1}}}}}$-concave,
 ${\rm{{\rm{H_tG_{h_1}}}}}$-concave)} and $g$ is
 ${\rm{{\rm{A_tG_{h_2}}}}}$-concave
\emph{( ${\rm{{\rm{G_tG_{h_2}}}}}$-concave,
 ${\rm{{\rm{H_tG_{h_2}}}}}$-concave)}, respectively; and $h\left(t\right):=min\{h_1\left(t\right),h_2\left(t\right)\}$,  where
 $h\left(t\right)+h\left(1-t\right) \ge c$. Then the
product $(fg)$ is  ${\rm{{\rm{A_tG_{c\cdot h}}}}}$-concave
\emph{( ${\rm{{\rm{G_tG_{c\cdot h}}}}}$-concave,
 ${\rm{{\rm{H_tG_{c\cdot h}}}}}$-concave)}, respectively.
\end{corollary}

\begin{proposition}
Let $f$ and $g$ be an oppositely  ordered functions. If $f$ is
 ${\rm{{\rm{A_tH_{h_1}}}}}$-convex
\emph{(${\rm{{\rm{G_tH_{h_1}}}}}$-convex,
 ${\rm{{\rm{H_tH_{h_1}}}}}$-convex)}, $g$ is
 ${\rm{{\rm{A_tH_{h_2}}}}}$-convex
\emph{($h_2$-${\rm{{\rm{G_tH_{h_2}}}}}$-convex,
 ${\rm{{\rm{H_tH_{h_2}}}}}$-convex)}, respectively; and
$h\left(t\right)+h\left(1-t\right) \ge c$, where
$h\left(t\right):=min\{h_1\left(t\right),h_2\left(t\right)\}$ and
$c$ is a fixed positive real number. Then the product $(fg)$ is
 ${\rm{{\rm{A_tH_{c\cdot h}}}}}$-convex
\emph{( ${\rm{{\rm{G_tH_{c\cdot h}}}}}$-convex,
 ${\rm{{\rm{H_tH_{c\cdot h}}}}}$-convex)}, respectively.
\end{proposition}
\begin{proof}
Since $f$ and $g$ are oppositely  ordered functions
\begin{align*}
f\left( x \right)g\left( x \right)+f\left( y \right)g\left( y
\right) \le f\left( x \right)g\left( y \right) +g\left( x
\right)f\left( y \right).
\end{align*}
Let $t$ and $s$ be positive numbers such that $t+s = 1$. Then we
obtain
\begin{align*}
 &\left( {fg} \right)\left( {A_t \left( {x,y} \right)} \right) \\
  &= \left( {fg} \right)\left( {sx + ty} \right) \\
  &\le \frac{{f\left( x \right)f\left( y \right)}}{{h_1 \left( t \right)f\left( x \right) + h_1 \left( s \right)f\left( y \right)}} \cdot \frac{{g\left( x \right)g\left( y \right)}}{{h_2 \left( t \right)g\left( x \right) + h_2 \left( s \right)g\left( y \right)}} \\
  &\le \frac{{\left( {fg} \right)\left( x \right)\left( {fg} \right)\left( y \right)}}{{h_1 \left( t \right)h_2 \left( t \right)f\left( x \right)g\left( x \right) + h_1 \left( s \right)h_2 \left( t \right)f\left( y \right)g\left( x \right) + h_1 \left( t \right)h_2 \left( s \right)f\left( x \right)g\left( y \right) + h_1 \left( s \right)h_2 \left( s \right)f\left( y \right)g\left( y \right)}} \\
  &\le \frac{{\left( {fg} \right)\left( x \right)\left( {fg} \right)\left( y \right)}}
  {{h^2 \left( t \right)f\left( x \right)g\left( x \right) + h\left( s \right)h\left( t \right)f\left( x \right)g\left( x \right)
  + h\left( t \right)h\left( s \right)f\left( y \right)g\left( y \right) + h^2 \left( s \right)f\left( y \right)g\left( y \right)}} \\
  &= \frac{{\left( {fg} \right)\left( x \right)\left( {fg} \right)\left( y \right)}}{{\left[ {h\left( t \right) + h\left( s \right)} \right]\left[ {h\left( t \right)\left( {fg} \right)\left( x \right) + h\left( s \right)\left( {fg} \right)\left( y \right)} \right]}} \\
  &= \frac{{\left( {fg} \right)\left( x \right)\left( {fg} \right)\left( y \right)}}{{c \cdot h\left( t \right)\left( {fg} \right)\left( x \right) + c \cdot h\left( s \right)\left( {fg} \right)\left( y
  \right)}}\\
  &=H\left( {c \cdot h(t);\left( {fg} \right)\left( x \right),\left( {fg}
\right)\left( y \right)} \right),
\end{align*}
which shows that $(fg)$ is  ${\rm{{\rm{A_tH_{c\cdot h}}}}}$-convex.  The cases when $fg$ is
 ${\rm{{\rm{G_tH_{c\cdot h}}}}}$-convex or  ${\rm{{\rm{H_tH_{c\cdot h}}}}}$-convex, are follow in similar manner.
\end{proof}

\begin{corollary}\label{cor4}
Let $f$ and $g$ be similarly  ordered functions. If $f$ is
 ${\rm{{\rm{A_tH_{h_1}}}}}$-concave
\emph{(${\rm{{\rm{G_tH_{h_1}}}}}$-concave,
 ${\rm{{\rm{H_tH_{h_1}}}}}$-concave)}, $g$ is
 ${\rm{{\rm{A_tH_{h_1}}}}}$-concave
\emph{($h_2$-${\rm{{\rm{G_tH_{h_2}}}}}$-concave,
 ${\rm{{\rm{H_tH_{h_2}}}}}$-concave)}, respectively; and
$h\left(t\right)+h\left(1-t\right) \le c$, where
$h\left(t\right):=max\{h_1\left(t\right),h_2\left(t\right)\}$ and
$c$ is a fixed positive real number. Then the product $(fg)$ is
 ${\rm{{\rm{A_tH_{c\cdot h}}}}}$-concave
\emph{( ${\rm{{\rm{G_tH_{c\cdot h}}}}}$-concave,
 ${\rm{{\rm{H_tH_{c\cdot h}}}}}$-concave)}, respectively.
\end{corollary}

Sometimes we often use functional inequalities to describe and
characterize all real functions that satisfy specific functional
inequality. In \cite{V}, Varo\v{s}anec proved a result regarding
${\rm{A_tA_h}}$-convex functions, following a similar approach; we
next present some results of this type.
\begin{theorem}\label{thm2}
Let $I \subset \mathbb{R}$ with $0 \in I$. Let $h$ be a
non-negative function on $J$.
\begin{enumerate}
\item Let $f$ be  ${\rm{A_tG_h}}$-convex and $f(0) = 1$. If $h$
is supermultiplicative, then the inequality
\begin{align}
f\left( {\alpha x + \beta y} \right)\le\left[ {f\left( x \right)}
\right]^{h\left( \alpha  \right)} \left[ {f\left( y \right)}
\right]^{h\left( \beta  \right)},\label{eq2.13}
\end{align}
holds for all $x,y \in I$ and all $\alpha, \beta> 0$ such that
$\alpha + \beta \le 1$.

\item Assume that $h\left(\alpha\right)< \frac{1}{2}$ for some
$\alpha \in \left(0,\frac{1}{2}\right)$. If $f$ is a non-negative
function such that inequality \eqref{eq2.13} holds for all $x,y
\in I$ and all $\alpha, \beta> 0$ such that $\alpha + \beta \le
1$, then $f(0) = 1$.

\item Let $f$ be  ${\rm{A_tG_h}}$-concave and $f(0) = 1$. If
$h$ is submultiplicative, then the inequality
\begin{align}
f\left( {\alpha x + \beta y} \right)\ge\left[ {f\left( x \right)}
\right]^{h\left( \alpha  \right)} \left[ {f\left( y \right)}
\right]^{h\left( \beta  \right)},\label{eq2.14}
\end{align}
holds for all $x,y \in I$ and all $\alpha, \beta> 0$ such that
$\alpha + \beta \le 1$.

\item Assume that $h\left(\alpha\right)> \frac{1}{2}$ for some
$\alpha \in \left(0,\frac{1}{2}\right)$. If $f$ is a non-negative
function such that inequality \eqref{eq2.14} holds for all $x,y
\in I$ and all $\alpha, \beta> 0$ such that $\alpha + \beta \le
1$, then $f(0) = 1$.
\end{enumerate}
\end{theorem}

\begin{proof}
Let $\alpha, \beta> 0$  be positive real numbers such that $\alpha
+ \beta =\lambda\le 1$.
\begin{enumerate}
    \item Define numbers $a$ and $b$ such as
    $a=\frac{\alpha}{\lambda}$ and $b=\frac{\beta}{\lambda}$. Then $a +b = 1$ and we have the following:
\begin{align*}
 f\left( {\alpha x + \beta y} \right) &= f\left( {\lambda ax + \lambda by}
 \right)\\
 &\le \left[ {f\left( {\lambda x} \right)} \right]^{h\left( a \right)} \left[ {f\left( {\lambda y} \right)} \right]^{h\left( b \right)}  \\
  &= \left[ {f\left( {\lambda x + \left( {1 - \lambda } \right) \cdot 0} \right)} \right]^{h\left( a \right)} \left[ {f\left( {\lambda y + \left( {1 - \lambda } \right) \cdot 0} \right)} \right]^{h\left( b \right)}  \\
  &\le \left\{ {\left[ {f\left( x \right)} \right]^{h\left( \lambda  \right)} \left[ {f\left( 0 \right)} \right]^{h\left( {1 - \lambda } \right)} } \right\}^{h\left( a \right)} \left\{ {\left[ {f\left( y \right)} \right]^{h\left( \lambda  \right)} \left[ {f\left( 0 \right)} \right]^{h\left( {1 - \lambda } \right)} } \right\}^{h\left( b \right)}  \\
  &= \left[ {f\left( x \right)} \right]^{h\left( a \right)h\left( \lambda  \right)} \left[ {f\left( y \right)} \right]^{h\left( b \right)h\left( \lambda  \right)}  \\
  &= \left[ {f\left( x \right)} \right]^{h\left( {\lambda a} \right)} \left[ {f\left( y \right)} \right]^{h\left( {\lambda b} \right)}  \\
  &= \left[ {f\left( x \right)} \right]^{h\left( \alpha  \right)} \left[ {f\left( y \right)} \right]^{h\left( \beta
  \right)},
 \end{align*}
where we use that $f$ is $A_tG_h$, $f(0) = 1$ and $h$ is
supermultiplicative, respectively.\\

\item Suppose that $f(0)\ne 1$. Putting $x = y = 0$ in
\eqref{eq2.13} we get
\begin{align*}
f\left( {0} \right) \le  \left[ {f\left( 0 \right)}
\right]^{h\left( \alpha  \right)+h\left( \beta  \right)} ,\qquad
{\text{for all}} \,\,\alpha, \beta> 0, \,\,\alpha + \beta \le 1.
\end{align*}
Setting $\beta=\alpha$, $\alpha\in \left(0,\frac{1}{2}\right)$,
then $0 \le \left( {2h\left( \alpha  \right) - 1} \right)\log
f\left( 0 \right)$, it follows that $h\left( \alpha  \right) \ge
\frac{1}{2}$, since $f(0)\ne 1$, which  contradicts the assumption
of theorem. So that $f(0)= 1$.
\end{enumerate}
The proofs for cases (3) and (4) are similar to the previous.
Hence, the proof is completely established.
\end{proof}

\begin{theorem}\label{thm3}
Let  $a,b \in \left( {\textstyle{1 \over \tau }} ,\infty\right)$
with $a<b$, so that $a,b\in I$ where $I=\left(0,\tau\right)$. Let
$h$ be a non-negative function on $J$.
\begin{enumerate}
\item Let $f$ be  ${\rm{G_tA_h}}$-convex and $f(1) = 0$. If $h$
is supermultiplicative, then the inequality
\begin{align}
f\left( {x^\alpha  y^\beta  } \right) \le  h\left( \alpha
\right)f\left( x \right) + h\left( \beta  \right)f\left( y
\right),\label{eq2.15}
\end{align}
holds for all $x,y \in I$ and all $\alpha, \beta> 0$ such that
$\alpha + \beta \le 1$.

\item Assume that $h\left(\alpha\right)< \frac{1}{2}$ for some
$\alpha \in \left(0,\frac{1}{2}\right)$. If $f$ is a non-negative
function such that inequality \eqref{eq2.15} holds for all $x,y
\in I$ and all $\alpha, \beta> 0$ such that $\alpha + \beta \le
1$, then $f(1) = 0$.

\item Let $f$ be  ${\rm{G_tA_h}}$-concave and $f(1) = 0$. If
$h$ is submultiplicative, then the inequality
\begin{align}
f\left( {x^\alpha  y^\beta  } \right) \ge  h\left( \alpha
\right)f\left( x \right) + h\left( \beta  \right)f\left( y
\right),\label{eq2.16}
\end{align}
holds for all $x,y \in I$ and all $\alpha, \beta> 0$ such that
$\alpha + \beta \le 1$.

\item Assume that $h\left(\alpha\right)> \frac{1}{2}$ for some
$\alpha \in \left(0,\frac{1}{2}\right)$. If $f$ is a non-negative
function such that inequality \eqref{eq2.16} holds for all $x,y
\in I$ and all $\alpha, \beta> 0$ such that $\alpha + \beta \le
1$, then $f(1) = 0$.
\end{enumerate}
\end{theorem}

\begin{proof}
Let $\alpha, \beta> 0$  be positive real numbers such that $\alpha
+ \beta =\lambda\le 1$.
\begin{enumerate}
    \item Define numbers $a$ and $b$ such as
    $a=\frac{\alpha}{\lambda}$ and $b=\frac{\beta}{\lambda}$. Then $a +b = 1$ and we have the following:
\begin{align*}
 f\left( {x^\alpha  y^\beta  } \right) &= f\left( {x^{\lambda a} y^{\lambda b} } \right) \\
 &\le h\left( a \right)f\left( {x^\lambda  } \right) + h\left( b \right)f\left( {y^\lambda  } \right) \\
  &= h\left( a \right)f\left( {x^\lambda   \cdot 1^{1 - \lambda } } \right) + h\left( b \right)f\left( {y^\lambda   \cdot 1^{1 - \lambda } } \right) \\
  &\le h\left( a \right)\left[ {h\left( \lambda  \right)f\left( x \right) + h\left( {1 - \lambda } \right)f\left( 1 \right)} \right] + h\left( b \right)\left[ {h\left( \lambda  \right)f\left( y \right) + h\left( {1 - \lambda } \right)f\left( 1 \right)} \right] \\
  &= h\left( a \right)h\left( \lambda  \right)f\left( x \right) + h\left( b \right)h\left( \lambda  \right)f\left( y \right) \\
  &\le h\left( \alpha  \right)f\left( x \right) + h\left( \beta  \right)f\left( y
  \right),
 \end{align*}
where we use that $f$ is ${\rm {G_tA_h}}$, $f(1) = 0$ and $h$ is
supermultiplicative, respectively.\\

\item Suppose that $f(1)\ne 0$, since $f$ is non-negative then $f
(1)
> 0$. Putting $x = y = 1$ in
\eqref{eq2.15} we get
\begin{align*}
f\left( {1} \right) \le   h\left( \alpha \right)f\left( 1 \right)
+ h\left( \beta  \right)f\left( 1 \right), \qquad {\text{ for
all}} \,\,\alpha, \beta> 0, \,\,\alpha + \beta \le 1.
\end{align*}
Setting $\beta=\alpha$, $\alpha\in \left(0,\frac{1}{2}\right)$,
then $0 \le \left( {2h\left( \alpha  \right) - 1} \right) f\left(
1 \right)$, it follows that $h\left( \alpha  \right) \ge
\frac{1}{2}$, which  contradicts the assumption of theorem. So
that $f(1)= 0$.
\end{enumerate}
The proofs for cases (3) and (4) are similar to the previous.
Hence, the proof is completely established.
\end{proof}

\begin{theorem}\label{thm4}
Let  $a,b \in \left( {\textstyle{1 \over \tau }} ,\infty\right)$
with $a<b$, so that $a,b\in I$ where $I=\left(0,\tau\right)$.  Let
$h$ be a non-negative function on $J$.
\begin{enumerate}
\item Let $f$ be  ${\rm{G_tG_h}}$-convex and $f(1) = 1$. If $h$
is supermultiplicative, then the inequality
\begin{align}
f\left( {x^\alpha  y^\beta  } \right) \le\left[ {f\left( x
\right)} \right]^{h\left( \alpha  \right)} \left[ {f\left( y
\right)} \right]^{h\left( \beta  \right)},\label{eq2.17}
\end{align}
holds for all $x,y \in I$ and all $\alpha, \beta> 0$ such that
$\alpha + \beta \le 1$.

\item Assume that $h\left(\alpha\right)< \frac{1}{2}$ for some
$\alpha \in \left(0,\frac{1}{2}\right)$. If $f$ is a non-negative
function such that inequality \eqref{eq2.17} holds for all $x,y
\in I$ and all $\alpha, \beta> 0$ such that $\alpha + \beta \le
1$, then $f(1) = 1$.

\item Let $f$ be $h$-${\rm{G_tG_h}}$-concave and $f(1) = 1$. If
$h$ is submultiplicative, then the inequality
\begin{align}
f\left( {x^\alpha  y^\beta  } \right) \ge\left[ {f\left( x
\right)} \right]^{h\left( \alpha  \right)} \left[ {f\left( y
\right)} \right]^{h\left( \beta  \right)},\label{eq2.18}
\end{align}
holds for all $x,y \in I$ and all $\alpha, \beta> 0$ such that
$\alpha + \beta \le 1$.

\item Assume that $h\left(\alpha\right)> \frac{1}{2}$ for some
$\alpha \in \left(0,\frac{1}{2}\right)$. If $f$ is a non-negative
function such that inequality \eqref{eq2.18} holds for all $x,y
\in I$ and all $\alpha, \beta> 0$ such that $\alpha + \beta \le
1$, then $f(1) = 1$.
\end{enumerate}
\end{theorem}

\begin{proof}
Let $\alpha, \beta> 0$  be positive real numbers such that $\alpha
+ \beta =\lambda\le 1$.
\begin{enumerate}
    \item Define numbers $a$ and $b$ such as
    $a=\frac{\alpha}{\lambda}$ and $b=\frac{\beta}{\lambda}$. Then $a +b = 1$ and we have the following:
\begin{align*}
 f\left( {x^\alpha  y^\beta  } \right)&= f\left( {x^{\lambda a} y^{\lambda b} }
 \right)\\
 &\le \left[ {f\left( {x^\lambda  } \right)} \right]^{h\left( a \right)} \left[ {f\left( {y^\lambda  } \right)} \right]^{h\left( b \right)}  \\
  &= \left[ {f\left( {x^\lambda   \cdot 1^{1 - \lambda } } \right)} \right]^{h\left( a \right)} \left[ {f\left( {y^\lambda   \cdot 1^{1 - \lambda } } \right)} \right]^{h\left( b \right)}  \\
  &\le \left\{ {\left[ {f\left( x \right)} \right]^{h\left( \lambda  \right)} \left[ {f\left( 1 \right)} \right]^{h\left( {1 - \lambda } \right)} } \right\}^{h\left( a \right)} \left\{ {\left[ {f\left( y \right)} \right]^{h\left( \lambda  \right)} \left[ {f\left( 1 \right)} \right]^{h\left( {1 - \lambda } \right)} } \right\}^{h\left( b \right)}  \\
  &= \left[ {f\left( x \right)} \right]^{h\left( a \right)h\left( \lambda  \right)} \left[ {f\left( y \right)} \right]^{h\left( b \right)h\left( \lambda  \right)}  \\
  &= \left[ {f\left( x \right)} \right]^{h\left( {\lambda a} \right)} \left[ {f\left( y \right)} \right]^{h\left( {\lambda b} \right)}  \\
  &= \left[ {f\left( x \right)} \right]^{h\left( \alpha  \right)} \left[ {f\left( y \right)} \right]^{h\left( \beta
  \right)},
 \end{align*}
where we use that $f$ is ${\rm {G_tG_h}}$, $f(1) = 1$ and $h$ is
supermultiplicative, respectively.\\

\item Suppose that $f(1)\ne 1$. Putting $x = y = 1$ in
\eqref{eq2.17} we get
\begin{align*}
f\left( {1} \right) \le \left[ {f\left( 1 \right)}
\right]^{h\left( \alpha  \right)} \left[ {f\left( 1 \right)}
\right]^{h\left( \beta  \right)}, \qquad {\text{ for all}}
\,\,\alpha, \beta> 0, \,\,\alpha + \beta \le 1.
\end{align*}
Setting $\beta=\alpha$, $\alpha\in \left(0,\frac{1}{2}\right)$,
then $1 \le  \left[f\left( 1 \right)\right]^{\left( {2h\left(
\alpha \right) - 1} \right)}$, it follows that $h\left( \alpha
\right) \ge \frac{1}{2}$,  which  contradicts the assumption of
theorem. So that $f(1)= 1$.
\end{enumerate}
The proofs for cases (3) and (4) are similar to the previous.
Hence, the proof is completely established.
\end{proof}

\subsection{Composition of $h$-${\rm{MN}}$-convex functions}

In the next three results, we assume the g $h_i:J_i \to
\left(0,\infty\right)$, $i=1,2$, $h_2\left(J_2\right)\subseteq
J_1$ are non-negative functions such that
$h_2\left(\alpha\right)+h_2\left(1-\alpha\right) \le 1$, for
$\alpha \left(0,1\right)\subseteq J_2$, let $f:I_1 \to
\left[0,\infty\right)$,  $g:I_2 \to \left[0,\infty\right)$, be
functions with $g\left(I_2\right)\subseteq I_1$.

\begin{theorem}\label{thm5}
Let  $f\left(1\right)=0$. If $h_1$ is a supermultiplicative
function, $f$ is ${\rm{G_tA_{h_1}}}$-convex and increasing
(decreasing) on $I_1$, while $g$ is  ${\rm{A_tG_{h_2}}}$-convex
(-concave) on $I_2$, then the composition $f\circ g$ is
 ${\rm{A_tA_{{h_1\circ h_2}}}}$-convex on $I_2$.

If $h_1$ is a submultiplicative function, $f$ is
 ${\rm{G_tA_{h_1}}}$-concave  and increasing (decreasing) on
$I_1$, while $g$ is  ${\rm{A_tG_{h_2}}}$-convex (-convex) on
$I_2$, then the composition $f\circ g$ is  ${\rm{A_tA_{h_1\circ h_2 }}}$-concave on $I_2$.
\end{theorem}

\begin{proof}
If $g$ is  ${\rm{A_tG_{h_2}}}$-convex on $I_2$ and $f$ increasing
then
\begin{align*}
f \circ g\left( {\alpha x + \left( {1 - \alpha } \right)y} \right)
\le f\left( {\left[ {g\left( x \right)} \right]^{^{h_2 \left(
\alpha  \right)} } \left[ {g\left( y \right)} \right]^{h_2 \left(
{1 - \alpha } \right)} } \right),
\end{align*}
for all $x,y \in I_2$ and $\alpha \in (0,1)$. Using Theorem
\ref{thm3}(1), we obtain that

\begin{align*}
f\left( {\left[ {g\left( x \right)} \right]^{^{h_2 \left( \alpha
\right)} } \left[ {g\left( y \right)} \right]^{h_2 \left( {1 -
\alpha } \right)} } \right)
  &\le h_1 \left( {h_2 \left( \alpha  \right)} \right)f\left( {g\left( x \right)} \right) + h_1 \left( {h_2 \left( {1 - \alpha } \right)} \right)f\left( {g\left( y \right)} \right) \\
  &= \left( {h_1  \circ h_2 } \right)\left( \alpha  \right)\left( {f \circ g} \right)\left( x \right) + \left( {h_1  \circ h_2 } \right)\left( {1 - \alpha } \right)\left( {f \circ g} \right)\left( y
  \right),
\end{align*}
which means that $f\circ g$ is  ${\rm{A_tA_{h_1\circ h_2}}}$-convex on $I_2$.
\end{proof}

\begin{theorem}\label{thm6}
Let  $0\in I_1$ and $f\left(0\right)=1$. If $h_1$ is a
supermultiplicative function, $f$ is  ${\rm{A_tG_{h_1}}}$-convex
and increasing (decreasing) on $I_1$, while $g$ is
${\rm{G_tA_{h_2}}}$-convex (-concave) on $I_2$, then the
composition $f\circ g$ is  ${\rm{G_tG_{h_1\circ h_2}}}$-convex on $I_2$.

If $h_1$ is a submultiplicative function, $f$ is
${\rm{A_tG_{h_1}}}$-concave  and increasing (decreasing) on
$I_1$, while $g$ is  ${\rm{G_tA_{h_2}}}$-convex (-convex) on
$I_2$, then the composition $f\circ g$ is  ${\rm{G_tG_{h_1\circ h_2}}}$-concave on $I_2$.
\end{theorem}

\begin{proof}
The proof is similar to the proof of Theorem \ref{thm5} and using
Theorem \ref{thm2}(1).
\end{proof}

\begin{theorem}\label{cor5}
Let    $f\left(1\right)=1$. If $h_1$ is a supermultiplicative
function, $f$ is  ${\rm{G_tG_{h_1}}}$-convex and increasing
(decreasing) on $I_1$, while $g$ is ${\rm{G_tG_{h_2}}}$-convex
(-concave) on $I_2$, then the composition $f\circ g$ is
${\rm{G_tG_{h_1\circ h_2}}}$-convex on $I_2$.

If $h_1$ is a submultiplicative function, $f$ is
${\rm{G_tG_t}}$-concave  and increasing (decreasing) on
$I_1$, while $g$ is ${\rm{G_tG_{h_2}}}$-convex (-convex) on
$I_2$, then the composition $f\circ g$ is ${\rm{G_tG_{h_1\circ h_2}}}$-concave on $I_2$.
\end{theorem}
\begin{proof}
The proof is similar to the proof of Theorem \ref{thm5} and using
Theorem \ref{thm4}(1).
\end{proof}

Next, we examine functions compositions, one of them is of type
 ${\rm{M_tK_{h_1}}}$-convex while the other is
${\rm{K_tN_{h_2}}}$-convex.
\begin{theorem}\label{thm7}
Let $M,N$ and $K$ be three mean functions. Let  $h_1:J_1 \to
\left(0,\infty\right)$ and $h_1:J_2 \to \left(0,1\right)$,
$h_2\left(J_2\right)\subseteq \left(0,1\right) \subseteq J_1$ are
non-negative functions  for $\alpha\in \left(0,1\right)\subseteq
J_2$ and $h_2\left(\alpha\right)\in \left(0,1\right)\subseteq
J_1$, let $f:I_1 \to \left[0,\infty\right)$,  $g:I_2 \to
\left[0,\infty\right)$, be functions with
$g\left(I_2\right)\subseteq I_1$. If  $f$ is
${\rm{K_tN_{h_1}}}$-convex  and increasing (decreasing) on
$I_1$, while $g$ is  ${\rm{M_tK_{h_2}}}$-convex (-concave) on
$I_2$, then the composition $f\circ g$ is  ${\rm{M_tN_{h_1\circ h_2}}}$-convex on $I_2$. Namely, we explore
this corollary in the table below.
\end{theorem}

\begin{proof}
We select to prove one of the mentioned cases and the others
follow in similar fashion. For example, if $g$ is
${\rm{H_tA_{h_2}}}$-convex on $I_2$ and $f$ is increasing then
\begin{align*}
f \circ g\left( { \frac{{xy}}{{\alpha x + \left( {1 - \alpha }
\right)y}} } \right) \le f\left( {h_2 \left( 1-\alpha
\right)g\left( x \right) + h_2 \left( { \alpha } \right)g\left( y
\right)} \right),
\end{align*}
for all $x,y \in I_2$ and $\alpha \in (0,1)$. Using Definition
\ref{def5}, we obtain that
\begin{align*}
f\left( {h_2 \left( 1-\alpha \right)g\left( x \right) + h_2 \left(
{ \alpha } \right)g\left( y \right)} \right)
  &\le
\frac{{f\left( {g\left( x \right)} \right)f\left( {g\left( y
\right)} \right)}}{{h_1 \left( {h_2 \left( {  \alpha } \right)}
\right)f\left( {g\left( x \right)} \right) + h_1 \left( {h_2
\left( 1-\alpha  \right)} \right)f\left( {g\left( y \right)}
\right)}}
\\
&=\frac{{\left( {f \circ g} \right)\left( x \right)\left( {f \circ
g} \right)\left( y \right)}}{{\left( {h_1  \circ h_2 }
\right)\left( {  \alpha } \right)\left( {f \circ g} \right)\left(
x \right) + \left( {h_1  \circ h_2 } \right)\left( 1-\alpha
\right)\left( {f \circ g} \right)\left( y \right)}},
\end{align*}
for $h_2\left(\alpha\right)\in \left(0,1\right)$, which shows that
$f\circ g$ is ${\rm{H_tH_{h_1\circ h_2}}}$-convex
on $I_2$.
\end{proof}
\begin{center}
\begin{tabular}{|c|c|c|c|}
  \hline
  $f$ & $g$ & $f\circ g$  \\
  \hline
  ${\rm{A_tA_{h_1}}}$-convex &  ${\rm{A_tA_{h_2}}}$-convex   &   \\
  ${\rm{G_tA_{h_1}}}$-convex &  ${\rm{A_tG_{h_2}}}$-convex   &   ${\rm{A_tA_{ h_1\circ h_2}}}$-convex \\
  ${\rm{H_tA_{h_1}}}$-convex &  ${\rm{A_tH_{h_2}}}$-convex   &  \\
  \hline
  ${\rm{A_tG_{h_1}}}$-convex &  ${\rm{A_tA_{h_2}}}$-convex   &   \\
  ${\rm{G_tG_{h_1}}}$-convex &  ${\rm{A_tG_{h_2}}}$-convex   &  ${\rm{A_tG_{ h_1\circ h_2}}}$-convex \\
  ${\rm{H_tG_{h_1}}}$-convex &  ${\rm{A_tH_{h_2}}}$-convex   &  \\
  \hline
  ${\rm{A_tH_{h_1}}}$-convex &  ${\rm{A_tA_{h_2}}}$-convex   &   \\
  ${\rm{G_tH_{h_1}}}$-convex &  ${\rm{A_tG_{h_2}}}$-convex   &  ${\rm{A_tH_{ h_1\circ h_2}}}$-convex \\
  ${\rm{H_tH_{h_1}}}$-convex &  ${\rm{A_tH_{h_2}}}$-convex   &  \\
  \hline
  ${\rm{A_tA_{h_1}}}$-convex &  ${\rm{G_tA_{h_2}}}$-convex   &   \\
  ${\rm{G_tA_{h_1}}}$-convex &  ${\rm{G_tG_{h_2}}}$-convex   &  ${\rm{G_tA_{ h_1\circ h_2}}}$-convex \\
  ${\rm{H_tA_{h_1}}}$-convex &  ${\rm{G_tH_{h_2}}}$-convex   &  \\
  \hline
  ${\rm{G_tG_{h_1}}}$-convex &  ${\rm{G_tG_{h_2}}}$-convex   &   \\
  ${\rm{A_tG_{h_1}}}$-convex &  ${\rm{G_tA_{h_2}}}$-convex   &  ${\rm{G_tG_{ h_1\circ h_2}}}$-convex \\
  ${\rm{H_tG_{h_1}}}$-convex &  ${\rm{G_tH_{h_2}}}$-convex   &  \\
  \hline
  ${\rm{A_tH_{h_1}}}$-convex &  ${\rm{G_tA_{h_2}}}$-convex   &   \\
  ${\rm{G_tH_{h_1}}}$-convex &  ${\rm{G_tG_{h_2}}}$-convex   &  ${\rm{G_tH_{ h_1\circ h_2}}}$-convex \\
  ${\rm{H_tH_{h_1}}}$-convex &  ${\rm{G_tH_{h_2}}}$-convex   &  \\
  \hline
  ${\rm{A_tA_{h_1}}}$-convex &  ${\rm{H_tA_{h_2}}}$-convex   &   \\
  ${\rm{G_tA_{h_1}}}$-convex &  ${\rm{H_tG_{h_2}}}$-convex   &   ${\rm{H_tA_{h_1\circ h_2}}}$-convex \\
  ${\rm{H_tA_{h_1}}}$-convex &  ${\rm{H_tH_{h_2}}}$-convex   &  \\
  \hline
  ${\rm{A_tG_{h_1}}}$-convex &  ${\rm{H_tA_{h_2}}}$-convex   &   \\
  ${\rm{G_tG_{h_1}}}$-convex &  ${\rm{H_tG_{h_2}}}$-convex   &  ${\rm{H_tG_{h_1\circ h_2}}}$-convex \\
  ${\rm{H_tG_{h_1}}}$-convex &  ${\rm{H_tH_{h_2}}}$-convex   &  \\
    \hline
   ${\rm{H_tH_{h_1}}}$-convex &  ${\rm{H_tH_{h_2}}}$-convex   &   \\
   ${\rm{A_tH_{h_1}}}$-convex &  ${\rm{H_tA_{h_2}}}$-convex   &   ${\rm{H_tH_{ h_1\circ h_2}}}$-convex \\
   ${\rm{G_tH_{h_1}}}$-convex &  ${\rm{H_tG_{h_2}}}$-convex   &  \\
  \hline
  \end{tabular}

\end{center}

\section{Characterization of $h$-${\rm{M_tN_t}}$-convexity }

Let $h:J\to [0,\infty)$ be a non-negative function and let $f : I
\to \mathbb{R}$ be a function. For all points $x_1,x_2,x_3 \in I$,
$x_1<x_2<x_3$ such that $x_2-x_1$, $x_3-x_2$ and $x_3-x_1$ in $J$.
In \cite{V}, Varo\v{s}anec proved that  if $h$ is
supermultiplicative, and
  $f$ is  ${\rm{A_tA_h}}$-convex function, then the
inequality
\begin{align*}
h\left( {x_3  - x_2 } \right) f\left( {x_1 } \right)+h\left( {x_2
- x_1 } \right)f\left( {x_3 } \right) \ge  h\left( {x_3  - x_1 }
\right)f\left( {x_2 } \right),
\end{align*}
holds. Also, if $h$ is submultiplicative, and
 $f$ is   ${\rm{A_tA_h}}$-convex function, then the
above inequality is reversed. In what follows, similar results for
${\rm{M_tN_h}}$-convex functions are proved.
\begin{theorem}\label{thm8}
Let $h:J\to [0,\infty)$ be a non-negative
 function and let $f : I \to \mathbb{R}$ be a
function. For all points $x_1,x_2,x_3 \in I$, $x_1<x_2<x_3$ such
that $x_2-x_1$, $x_3-x_2$ and $x_3-x_1$ in $J$,
\begin{enumerate}
    \item  If $h$ is supermultiplicative, and
  $f$ is  ${\rm{A_tG_h}}$-convex function, then the
following inequality hold:
\begin{align*}
\left[ {f\left( {x_1 } \right)} \right]^{h\left( {x_3  - x_2 }
\right)} \left[ {f\left( {x_3 } \right)} \right]^{h\left( {x_2  -
x_1 } \right)}  \ge \left[ {f\left( {x_2 } \right)}
\right]^{h\left( {x_3  - x_1 } \right)},
\end{align*}

\item  If $h$ is submultiplicative, and
 $f$ is   ${\rm{A_tH_h}}$-convex function, then the
following inequality hold:
 \begin{align*}
h\left( {x_3 -x_1 } \right)f\left( x_1 \right)f\left( x_3 \right)
\ge h\left( {x_2 -x_1} \right)f\left( x_1 \right) f\left( x_2
\right) + h\left( {x_3 -x_2 } \right)f\left( x_3 \right)f\left(
x_2 \right).
\end{align*}
\end{enumerate}
In case of ${\rm{A_tN_h}}$-concavity the inequalities are
reversed.
\end{theorem}

\begin{proof} Let $x_1,x_2,x_3 \in I$ with $x_1<x_2<x_3$, such
that $x_2-x_1$, $x_3-x_2$ and $x_3-x_1$ in $J$. Consequently,
$\frac{x_2-x_1}{x_3-x_1}, \frac{x_3-x_2}{x_3-x_1} \in
\left(0,1\right)\subseteq J$ and $\frac{x_2-x_1}{x_3-x_1}+
\frac{x_3-x_2}{x_3-x_1} =1$. Also, since $h$ is
super(sub)multiplicative then for all $p,q\in J$ we have
\begin{align*}
h\left( {p } \right) = h\left( {\frac{{p }}{{q }} \cdot q }
\right) \ge (\le) h\left( {\frac{{p}}{{q }}} \right)h\left( {q }
\right),
\end{align*}
and this yield that
\begin{align*}
\frac{h\left( {p } \right)}{h\left( {q } \right)}  \ge (\le)
h\left( {\frac{{p}}{{q }}} \right).
\end{align*}
Setting $t = \frac{{x_3  - x_2 }}{{x_3  - x_1 }}$, $\alpha = x_1$,
$\beta  = x_3$, therefore we have the following cases:
\begin{enumerate}
\item  For $x_2  = t\alpha  + \left( {1 - t} \right)\beta$ and
since $f$ is ${\rm{A_tG_h}}$-convex, then by \eqref{eqhAG}
\begin{align}
f\left( x_2 \right) &\le \left[ {f\left( x_1  \right)}
\right]^{h\left(\frac{{x_3  - x_2}}{{x_3  - x_1 }}\right)} \left[
{f\left( x_3 \right)} \right]^{h\left(\frac{{x_2  - x_1}}{{x_3  -
x_1 }}\right)}\nonumber
\\
&\le \left[ {f\left( x_1  \right)} \right]^{\frac{{h\left(x_3  -
x_2\right)}}{{h\left(x_3  - x_1 \right)}}} \left[ {f\left( x_3
\right)} \right]^{\frac{{h\left(x_2  - x_1\right)}}{{h\left(x_3  -
x_1 \right)}}}, \label{eq3.1}
\end{align}
since $f$ is positive,  then the above inequality equivalent to
\begin{align*}
h\left( {x_3  - x_1 } \right)\log f\left( {x_2 } \right) \le
h\left( {x_3  - x_2 } \right)\log f\left( {x_1 } \right) + h\left(
{x_2  - x_1 } \right)\log f\left( {x_3 } \right).
\end{align*}
Rearranging the terms again we get
\begin{align*}
\left[ {f\left( {x_1 } \right)} \right]^{h\left( {x_3  - x_2 }
\right)} \left[ {f\left( {x_3 } \right)} \right]^{h\left( {x_2  -
x_1 } \right)}  \ge  \left[ {f\left( {x_2 } \right)}
\right]^{h\left( {x_3  - x_1 } \right)},
\end{align*}
as desired.

\item For $x_2  = t\alpha  + \left( {1 - t} \right)\beta$ and
since $f$ is ${\rm{A_tH_h}}$-convex then by \eqref{eqhAH}
 \begin{align}
f\left( x_2 \right) &\le \frac{{ f\left( x_1  \right)f\left( x_3
\right)}}{{h\left( {\frac{x_2 -x_1}{x_3-x_1}} \right)f\left( x_1
\right) + h\left( {\frac{x_3 -x_2}{x_3-x_1} } \right)f\left( x_3
\right)}}\nonumber
\\
&\le \frac{{ h\left( {x_3 -x_1 } \right)f\left( x_1 \right)f\left(
x_3 \right)}}{{h\left( {x_2 -x_1} \right)f\left( x_1 \right) +
h\left( {x_3 -x_2 } \right)f\left( x_3 \right)}},\label{eq3.2}
\end{align}
and this is equivalent to write
 \begin{align*}
h\left( {x_3 -x_1 } \right)f\left( x_1 \right)f\left( x_3 \right)
\ge h\left( {x_2 -x_1} \right)f\left( x_1 \right) f\left( x_2
\right) + h\left( {x_3 -x_2 } \right)f\left( x_3 \right)f\left(
x_2 \right),
\end{align*}
as desired.
\end{enumerate}
Thus, the proof is completely established.
\end{proof}

\begin{corollary}\label{cor6}
Let $h:(0,1)\to [0,\infty)$ be a non-negative
 function and let $f : (0,1) \to \left(0,\infty\right)$ be a
function. For all points $x_1,x_2,x_3 \in (0,1)$, $x_1<x_2<x_3$
such that $x_2-x_1$, $x_3-x_2$ and $x_3-x_1$ in $(0,1)$. Let
$h_r(t)= t^r$, $r \in \left(-\infty,-1\right]\cup
\left[0,1\right]$.
\begin{enumerate}
    \item  If
  $f$ is   ${\rm{A_tG_h}}$-convex function, then the
following inequality hold:
\begin{align*}
\left[ {f\left( {x_1 } \right)} \right]^{ \left( {x_3  - x_2 }
\right)^r} \left[ {f\left( {x_3 } \right)} \right]^{ \left( {x_2 -
x_1 } \right)^r}  \ge \left[ {f\left( {x_2 } \right)}
\right]^{\left( {x_3  - x_1 } \right)^r}.
\end{align*}
Furthermore, if $f(x)=x^{\lambda}$ $(\lambda < 0)$ we get several
Schur type inequalities.\\

\item  If  $f$ is  ${\rm{A_tH_h}}$-convex function, then the
following inequality hold:
 \begin{align*}
 \left( {x_3 -x_1 } \right)^rf\left( x_1 \right)f\left( x_3 \right)
\ge \left( {x_2 -x_1} \right)^rf\left( x_1 \right) f\left( x_2
\right) + \left( {x_3 -x_2 } \right)^r f\left( x_3 \right)f\left(
x_2 \right).
\end{align*}
Furthermore, if $f(x)=x^{\lambda}$ $(-1<\lambda <0)$ we get
several Schur type inequalities.
\end{enumerate}
In case of  ${\rm{A_tN_h}}$-concavity the inequalities are
reversed.
\end{corollary}

\begin{theorem}\label{thm9}
Let $h:J\to [0,\infty)$ be a non-negative
 function and let $f : I \to \mathbb{R}$ be a
function. For all points $x_1,x_2,x_3 \in I$, $x_1<x_2<x_3$ such
that $\ln \left( {{\textstyle{{x_3 } \over {x_2 }}}} \right)$,
$\ln \left( {{\textstyle{{x_2 } \over {x_1 }}}} \right)$ and $\ln
\left( {{\textstyle{{x_3 } \over {x_1 }}}} \right)$ in $J$.
\begin{enumerate}
\item  If $h$ is supermultiplicative, and $f$ is
 ${\rm{G_tA_h}}$-convex function, then the following inequality
hold:
\begin{align*}
 h\left( {\ln \left( {{\textstyle{{x_3 } \over {x_2 }}}} \right)}
\right)\cdot f\left( x_1 \right) + h\left( {\ln \left(
{{\textstyle{{x_2 } \over {x_1 }}}} \right)} \right) \cdot
f\left(x_3 \right)\ge h\left( {\ln \left( {{\textstyle{{x_3 }
\over {x_1 }}}} \right)} \right) f\left( x_2 \right).
\end{align*}

\item  If $h$ is supermultiplicative, and $f$ is
 ${\rm{G_tG_h}}$-convex function, then the following inequality
hold:
\begin{align*} \left[ {f\left( x_1 \right)} \right]^{h\left(
{\ln \left( {{\textstyle{{x_3 } \over {x_1 }}}} \right)} \right) }
\left[ {f\left( x_3 \right)} \right]^{h\left( {\ln \left(
{{\textstyle{{x_2 } \over {x_1 }}}} \right)} \right)}\ge f\left(
x_2 \right)^{h\left( {\ln \left( {{\textstyle{{x_3 } \over {x_1
}}}} \right)} \right)}.
\end{align*}

\item  If $h$ is submultiplicative, and
 $f$ is   ${\rm{G_tH_h}}$-convex function, then the
following inequality hold:
\begin{align*}
h\left( {\ln \left(\frac{x_3}{x_1}\right)  } \right)f\left( x_1
\right)f\left( x_3 \right)+\ge h\left( {\ln
\left(\frac{x_2}{x_1}\right) } \right) f\left(x_1 \right)f\left(
x_2 \right) + h\left( {\ln \left(\frac{x_3}{x_2}\right) } \right)
f\left( x_3 \right)f\left( x_2 \right).
\end{align*}
\end{enumerate}
In case of ${\rm{G_tN_h}}$-concavity the inequalities are
reversed.
\end{theorem}

\begin{proof}
Let $x_1,x_2,x_3 \in I$ with $x_1<x_2<x_3$, such that  $\ln \left(
{{\textstyle{{x_3 } \over {x_2 }}}} \right)$, $\ln \left(
{{\textstyle{{x_2 } \over {x_1 }}}} \right)$ and $\ln \left(
{{\textstyle{{x_3 } \over {x_1 }}}} \right)$ in $J$. Consequently,
$\frac{\ln x_2-\ln x_1}{\ln x_3-\ln x_1}, \frac{\ln x_3-\ln
x_2}{\ln x_3-\ln x_1} \in \left(0,1\right)\subseteq J$ and
$\frac{\ln x_2-\ln x_1}{\ln x_3-\ln x_1}+ \frac{\ln x_3-\ln
x_2}{\ln x_3-\ln x_1}  =1$. Setting $t = \frac{\ln x_3-\ln
x_2}{\ln x_3-\ln x_1} $, $\alpha = x_1$, $\beta  = x_3$, therefore
we have the following cases:
\begin{enumerate}
\item For $x_2  = \alpha^t \beta^{1-t}$ and since $f$ is
${\rm{G_tA_h}}$-convex then by \eqref{eqhGA}
\begin{align}
f\left( x_2 \right) &\le h\left(\frac{{\ln \left( x_3 \right) -
\ln \left( x_2  \right)}}{{\ln \left( x_3  \right) - \ln \left(x_1
\right)}}\right) \cdot f\left( x_1  \right) +  h\left(\frac{{\ln
\left( x_2 \right) - \ln \left( x_1 \right)}}{{\ln \left( x_3
\right) - \ln \left( x_1 \right)}}\right) \cdot f\left(x_3 \right)
\nonumber\\
&\le \frac{{h\left(\ln \left( x_3 \right) - \ln \left( x_1
\right)\right)}}{{h\left(\ln \left( x_3  \right) - \ln \left(x_1
\right)\right)}} \cdot f\left( x_1  \right) +   \frac{{h\left(\ln
\left( x_2 \right) - \ln \left( x_1 \right)\right)}}{{h\left(\ln
\left( x_3  \right) - \ln \left(x_1 \right)\right)}} \cdot
f\left(x_3 \right), \label{eq3.3}
\end{align}
and this is equivalent to write
\begin{align*}
 h\left( {\ln \left( {{\textstyle{{x_3 } \over {x_2 }}}} \right)}
\right)\cdot f\left( x_1 \right) + h\left( {\ln \left(
{{\textstyle{{x_2 } \over {x_1 }}}} \right)} \right) \cdot
f\left(x_3 \right)\ge h\left( {\ln \left( {{\textstyle{{x_3 }
\over {x_1 }}}} \right)} \right) f\left( x_2 \right),
\end{align*}
as desired.

\item  For $x_2  = \alpha^t \beta^{1-t}$ and since $f$ is
${\rm{G_tG_h}}$-convex then by \eqref{eqhGG}
\begin{align}
f\left( x_2 \right) &\le \left[ {f\left( x_1  \right)}
\right]^{h\left(\frac{{\ln \left( x_3 \right) - \ln \left( x_2
\right)}}{{\ln \left( x_3  \right) - \ln \left( x_1
\right)}}\right)} \left[ {f\left( x_3  \right)}
\right]^{h\left(\frac{{\ln \left( x_2 \right) - \ln \left( x_1
\right)}}{{\ln \left( x_3 \right) - \ln \left( x_1
\right)}}\right)}
\nonumber\\
&\le \left[ {f\left( x_1  \right)} \right]^{\frac{{h\left(\ln
\left( x_3 \right) - \ln \left( x_2 \right)\right)}}{{h\left(\ln
\left( x_3 \right) - \ln \left( x_1 \right)\right)}}} \left[
{f\left( x_3 \right)} \right]^{\frac{{h\left(\ln \left( x_2
\right) - \ln \left( x_1 \right)\right)}}{{h\left(\ln \left( x_3
\right) - \ln \left( x_1 \right)\right)}}}, \label{eq3.4}
\end{align}
since $f$ is positive therefore
\begin{align*}
h\left( {\ln \left( {{\textstyle{{x_3 } \over {x_1 }}}} \right)}
\right) \log f\left( x_2 \right)\le h\left( {\ln \left(
{{\textstyle{{x_3 } \over {x_1 }}}} \right)} \right) \log \left[
{f\left( x_1 \right)} \right]+ h\left( {\ln \left(
{{\textstyle{{x_2 } \over {x_1 }}}} \right)} \right) \log\left[
{f\left( x_3 \right)} \right],
\end{align*}
and this equivalent to write
\begin{align*}
\left[ {f\left( x_1 \right)} \right]^{h\left( {\ln \left(
{{\textstyle{{x_3 } \over {x_1 }}}} \right)} \right) } \left[
{f\left( x_3 \right)} \right]^{h\left( {\ln \left(
{{\textstyle{{x_2 } \over {x_1 }}}} \right)} \right)}\ge f\left(
x_2 \right)^{h\left( {\ln \left( {{\textstyle{{x_3 } \over {x_1
}}}} \right)} \right)},
\end{align*}
as desired.

\item For $x_2  = \alpha^t \beta^{1-t}$ and since $f$ is
${\rm{G_tH_h}}$-convex then by \eqref{eqhGH}
\begin{align}
f\left( x_2 \right) &\le \frac{{f\left( x_1  \right)f\left( x_3
\right) }}{{h\left( {\frac{{\ln x_2  - \ln x_1 }}{{\ln x_3  - \ln
x_1 }}} \right) f\left(x_1  \right) + h\left( {\frac{{\ln x_3 -
\ln x_2 }}{{\ln x_3  - \ln x_1 }}} \right) f\left( x_3
\right)}}\nonumber
\\
&\le \frac{{h\left( {\ln x_3  - \ln x_1 } \right)f\left( x_1
\right)f\left( x_3 \right) }}{{ h\left( {\ln x_2  - \ln x_1 }
\right) f\left(x_1  \right) + h\left( {\ln x_3  - \ln x_2 }
\right) f\left( x_3 \right)}}, \label{eq3.5}
\end{align}
which is equivalent to write
\begin{align*}
h\left( {\ln \left(\frac{x_3}{x_1}\right)  } \right)f\left( x_1
\right)f\left( x_3 \right)- h\left( {\ln
\left(\frac{x_2}{x_1}\right) } \right) f\left(x_1 \right)f\left(
x_2 \right) \ge  h\left( {\ln \left(\frac{x_3}{x_2}\right) }
\right) f\left( x_3 \right)f\left( x_2 \right),
\end{align*}
as desired.
\end{enumerate}
Thus, the proof is completely established.
\end{proof}

\begin{corollary}\label{cor7}
Let $h:(0,1)\to [0,\infty)$ be a non-negative
 function and let $f : (0,1) \to \mathbb{R}$ be a
function. For all points $x_1,x_2,x_3 \in (0,1)$, $x_1<x_2<x_3$
such that $\ln \left( {{\textstyle{{x_3 } \over {x_2 }}}}
\right)$, $\ln \left( {{\textstyle{{x_2 } \over {x_1 }}}} \right)$
and $\ln \left( {{\textstyle{{x_3 } \over {x_1 }}}} \right)$ in
$(0,1)$. For $h_r(t)= t^r$, $r \in \left(-\infty,-1\right]\cup
\left[0,1\right]$.
\begin{enumerate}
\item  If   $f$ is  ${\rm{G_tA_h}}$-convex function, then the
following inequality hold:
\begin{align*}
  \left( {\ln \left( {{\textstyle{{x_3 } \over {x_2 }}}} \right)}
\right)^r\cdot f\left( x_1 \right) + \left( {\ln \left(
{{\textstyle{{x_2 } \over {x_1 }}}} \right)} \right)^r \cdot
f\left(x_3 \right)\ge \left( {\ln \left( {{\textstyle{{x_3 } \over
{x_1 }}}} \right)} \right)^r f\left( x_2 \right).
\end{align*}
Furthermore, if $f(x)=x^{\lambda}$ $(\lambda \in \mathbb{R})$ we
get several
Schur type inequalities.\\

\item  If $f$ is  ${\rm{G_tG_h}}$-convex function, then the
following inequality hold:
\begin{align*} \left[ {f\left( x_1 \right)} \right]^{ \left(
{\ln \left( {{\textstyle{{x_3 } \over {x_1 }}}} \right)} \right)^r
} \left[ {f\left( x_3 \right)} \right]^{\left( {\ln \left(
{{\textstyle{{x_2 } \over {x_1 }}}} \right)} \right)^r}\ge f\left(
x_2 \right)^{\left( {\ln \left( {{\textstyle{{x_3 } \over {x_1
}}}} \right)} \right)^r}.
\end{align*}

\item  If
 $f$ is   ${\rm{G_tH_h}}$-convex function, then the
following inequality hold:
\begin{align*}
\left( {\ln \left(\frac{x_3}{x_1}\right)  } \right)^rf\left( x_1
\right)f\left( x_3 \right)+\ge \left( {\ln
\left(\frac{x_2}{x_1}\right) } \right)^r f\left(x_1 \right)f\left(
x_2 \right) +  \left( {\ln \left(\frac{x_3}{x_2}\right) }
\right)^r f\left( x_3 \right)f\left( x_2 \right).
\end{align*}
\end{enumerate}
In case of  ${\rm{G_tN_h}}$-concavity the inequalities are
reversed.
\end{corollary}

\begin{theorem}\label{thm10}
Let $h:J\to [0,\infty)$ be a non-negative
 function and let $f : I \to \mathbb{R}$ be a
function. For all points $x_1,x_2,x_3 \in I$, $x_1<x_2<x_3$ such
that $x_1 \left(x_3  - x_2\right)$, $x_3 \left(x_2  - x_1\right)$
 and $x_2 \left(x_3  - x_1\right)$ in $J$,
\begin{enumerate}
\item  If $h$ is supermultiplicative, and $f$ is
 ${\rm{H_tA_h}}$-convex function, then the following inequality
hold:
\begin{align*}
h\left(x_1 \left(x_3  - x_2\right) \right)f\left( x_1 \right) +
h\left(x_3 \left( {x_2 - x_1 } \right)\right) f\left( x_3 \right)
\ge h\left(x_2\left( {x_3 - x_1 } \right)\right)f\left( x_2
\right),
\end{align*}

\item  If $h$ is supermultiplicative, and $f$ is
 ${\rm{H_tG_h}}$-convex function, then the following inequality
hold:
\begin{align*}
\left[ {f\left( x_1  \right)} \right]^{h\left(x_1\left( {x_3 -
x_2} \right)\right)} \cdot \left[ {f\left( x_3  \right)}
\right]^{h\left(x_3 \left( {x_2 - x_1 } \right)\right)} \ge \left[
f\left( x_2 \right)\right]^{h\left(x_2\left( {x_3 - x_1 }
\right)\right)},
\end{align*}

\item  If $h$ is submultiplicative, and
 $f$ is   ${\rm{H_tH_h}}$-convex function, then the
following inequality hold:
\begin{align*}
h\left(x_3 \left( {x_2 - x_1 } \right)\right)f\left( x_1
\right)f\left( x_2 \right) + h\left(x_1 \left( {x_3 - x_2}
\right)\right)f\left( x_2 \right) f\left( x_3 \right) \le
h\left(x_2\left( {x_3  - x_1 } \right)\right) f\left( x_1
\right)f\left( x_3 \right),
\end{align*}
\end{enumerate}
In case of  ${\rm{H_tN_h}}$-concavity the inequalities are
reversed.
\end{theorem}

\begin{proof}
Let $x_1,x_2,x_3 \in I$ with $x_1<x_2<x_3$, such that  $x_1
\left(x_3  - x_2\right), x_3 \left(x_2  - x_1\right), x_2
\left(x_3 - x_1\right) \in J$. And $\frac{{x_1 \left(x_3  -
x_2\right) }}{{x_2\left( {x_3 - x_1 } \right)}}  , \frac{{x_3
\left(x_2  - x_1\right) }}{{x_2\left( {x_3 - x_1 } \right)}} \in
\left(0,1\right)\subseteq J$, so that   $\frac{{x_1 \left(x_3  -
x_2\right) }}{{x_2\left( {x_3 - x_1 } \right)}} + \frac{{x_3
\left(x_2  - x_1\right) }}{{x_2\left( {x_3 - x_1 } \right)}} =1$.
 Setting $t = \frac{{x_3
\left(x_2  - x_1\right) }}{{x_2\left( {x_3 - x_1 } \right)}}$,
$\alpha = x_1$, $\beta  = x_3$, therefore we have the following
cases:
\begin{enumerate}
\item For  $x_2 = \frac{{\alpha \beta }}{{t\alpha + \left( {1 - t}
\right)\beta }} $  and since $f$ is ${\rm{H_tA_h}}$-convex then by
\eqref{eqhHA}
\begin{align}
f\left( x_2 \right) &\le h\left(\frac{{x_1 \left(x_3  - x_2\right)
}}{{x_2\left( {x_3 - x_1 } \right)}}\right)f\left( x_1 \right)
+h\left( \frac{{x_3 \left( {x_2 - x_1 } \right)}}{{x_2\left( {x_3
- x_1 } \right)}} \right)f\left( x_3 \right)
\nonumber\\
&\le \frac{{h\left(x_1 \left(x_3  - x_2\right)
\right)}}{{h\left(x_2\left( {x_3 - x_1 } \right)\right)}}f\left(
x_1 \right) + \frac{{h\left(x_3 \left( {x_2 - x_1 }
\right)\right)}}{{h\left(x_2\left( {x_3 - x_1 } \right)\right)}}
f\left( x_3 \right), \label{eq3.6}
\end{align}
which is equivalent to write
\begin{align*}
h\left(x_1 \left(x_3  - x_2\right) \right)f\left( x_1 \right) +
h\left(x_3 \left( {x_2 - x_1 } \right)\right) f\left( x_3 \right)
\ge h\left(x_2\left( {x_3 - x_1 } \right)\right)f\left( x_2
\right),
\end{align*}
as desired.

\item For  $x_2 = \frac{{\alpha \beta }}{{t\alpha + \left( {1 - t}
\right)\beta }} $  and since $f$ is ${\rm{H_tG_h}}$-convex then by
\eqref{eqhHG}
\begin{align}
f\left( x_2 \right) &\le  \left[ {f\left( x_1  \right)}
\right]^{h\left(\frac{{x_1\left( {x_3  - x_2} \right)}}{{x_2\left(
{x_3 - x_1 } \right)}}\right)} \left[ {f\left( x_3  \right)}
\right]^{h\left(\frac{{x_3 \left( {x_2 - x_1 }
\right)}}{{x_2\left( {x_3 - x_1} \right)}}\right)}
\nonumber\\
&\le  \left[ {f\left( x_1  \right)}
\right]^{\frac{{h\left(x_1\left( {x_3 - x_2}
\right)\right)}}{{h\left(x_2\left( {x_3 - x_1 } \right)\right)}}}
\left[ {f\left( x_3  \right)} \right]^{\frac{{h\left(x_3 \left(
{x_2 - x_1 } \right)\right)}}{{h\left(x_2\left( {x_3 - x_1}
\right)\right)}}},\label{eq3.7}
\end{align}
and this equivalent to write
\begin{align*}
\left[ {f\left( x_1  \right)} \right]^{h\left(x_1\left( {x_3 -
x_2} \right)\right)} \cdot \left[ {f\left( x_3  \right)}
\right]^{h\left(x_3 \left( {x_2 - x_1 } \right)\right)} \ge \left[
f\left( x_2 \right)\right]^{h\left(x_2\left( {x_3 - x_1 }
\right)\right)},
\end{align*}
as desired.

\item For  $x_2 = \frac{{\alpha \beta }}{{t\alpha + \left( {1 - t}
\right)\beta }} $  and since $f$ is ${\rm{H_tH_h}}$-convex then by
\eqref{eqhHH}
\begin{align}
f\left( x_2 \right) &\le \frac{{ f\left( x_1  \right)f\left( x_3
\right)}}{{h\left(\frac{x_3 \left( {x_2 - x_1 } \right)}{x_2\left(
{x_3  - x_1 } \right)}\right)f\left( x_1 \right) +
h\left(\frac{x_1 \left( {x_3 - x_2} \right)}{x_2\left( {x_3  - x_1
} \right)}\right)f\left( x_3 \right)}}
\nonumber\\
&\le \frac{{h\left(x_2\left( {x_3  - x_1 } \right)\right) f\left(
x_1  \right)f\left( x_3 \right)}}{{h\left(x_3 \left( {x_2 - x_1 }
\right)\right)f\left( x_1 \right) + h\left(x_1 \left( {x_3 - x_2}
\right)\right)f\left( x_3 \right)}}, \label{eq3.8}
\end{align}
and this equivalent to write
\begin{align*}
h\left(x_3 \left( {x_2 - x_1 } \right)\right)f\left( x_1
\right)f\left( x_2 \right) + h\left(x_1 \left( {x_3 - x_2}
\right)\right)f\left( x_2 \right) f\left( x_3 \right) \le
h\left(x_2\left( {x_3  - x_1 } \right)\right) f\left( x_1
\right)f\left( x_3 \right),
\end{align*}
as desired.
\end{enumerate}
Thus, the proof is completely established.
\end{proof}

\begin{corollary}\label{cor8}
Let $h:(0,1)\to [0,\infty)$ be a non-negative
 function and let $f : (0,1) \to \mathbb{R}$ be a
function. For all points $x_1,x_2,x_3 \in (0,1)$, $x_1<x_2<x_3$
such that $x_1 \left(x_3  - x_2\right)$, $x_3 \left(x_2  -
x_1\right)$
 and $x_2 \left(x_3  - x_1\right)$ in $(0,1)$.
For $h_r(t)= t^r$, $r \in \left(-\infty,-1\right]\cup
\left[0,1\right]$.
\begin{enumerate}
\item  If  $f$ is ${\rm{H_tA_h}}$-convex function, then the
following inequality hold:
\begin{align*}
 \left(x_1 \left(x_3  - x_2\right) \right)^rf\left( x_1 \right) +
\left(x_3 \left( {x_2 - x_1 } \right)\right)^r f\left( x_3 \right)
\ge \left(x_2\left( {x_3 - x_1 } \right)\right)^rf\left( x_2
\right).
\end{align*}
Furthermore, if $f(x)=x^{\lambda}$ $(\lambda >0 )$ we get several
Schur type inequalities.\\

\item  If  $f$ is $h$-${\rm{H_tG_t}}$-convex function, then the
following inequality hold:
\begin{align*}
\left[ {f\left( x_1  \right)} \right]^{\left(x_1\left( {x_3 - x_2}
\right)\right)^r} \cdot \left[ {f\left( x_3  \right)}
\right]^{\left(x_3 \left( {x_2 - x_1 } \right)\right)^r} \ge
\left[ f\left( x_2 \right)\right]^{\left(x_2\left( {x_3 - x_1 }
\right)\right)^r}.
\end{align*}

\item  If  $f$ is  ${\rm{H_tH_h}}$-convex function, then the
following inequality hold:
\begin{align*}
\left(x_3 \left( {x_2 - x_1 } \right)\right)^rf\left( x_1
\right)f\left( x_2 \right) + \left(x_1 \left( {x_3 - x_2}
\right)\right)^rf\left( x_2 \right) f\left( x_3 \right) \le
\left(x_2\left( {x_3  - x_1 } \right)\right)^r f\left( x_1
\right)f\left( x_3 \right).
\end{align*}
Furthermore, if $f(x)=x^{\lambda}$ $(1>\lambda > 0)$ we get
several Schur type inequalities.
\end{enumerate}
In case of ${\rm{H_tN_h}}$-concavity the inequalities are
reversed.
\end{corollary}

\begin{remark}
In \cite{MP}, Mitrinovi\'{c} and Pe\v{c}ari\'{c} proved the
validity of the inequality
\begin{align*}
\left( {x_1  - x_2 } \right)\left( {x_1  - x_3 } \right)f\left(
{x_1 } \right) + \left( {x_2  - x_1 } \right)\left( {x_2  - x_3 }
\right)f\left( {x_2 } \right) + \left( {x_3  - x_1 } \right)\left(
,{x_3  - x_2 } \right)f\left( {x_3 } \right) \ge 0
\end{align*}
for all $x_1,x_2,x_3 \in (0,1)$ and $f \in Q(I)$. Moreover, if
$f(x)=x^{\lambda}$ $(\lambda \in \mathbb{R})$, then the inequality
is of Schur type, see (\cite{MPF}, p.117). A similar inequality
for monotone convex functions was proved by Wright in \cite{W}. A
generalization to $h$-convex type functions was also presented in
\cite{V}.

In Corollaries \ref{cor6}--\ref{cor8}, if we choose $r=-1$, i.e.,
$h\left(x\right)=x^{-1}$, then several inequalities for
${\rm{M_tN_h}}$-convex functions can be deduced. For inequalities
of Schur type choose $f(x)=x^{\lambda}$
$(\lambda\in\mathbb{\mathbb{R}})$, taking into account that some
additional assumption on $\lambda$ have to be made to guarantee
the ${\rm{M_tN_h}}$-convexity of $f$.
\end{remark}

\section{Jensen's  type inequalities}
The weighted Arithmetic, Geometric, and Harmonic Means for
$n$-points $x_1,x_2,\cdots, x_n$ $(n\ge2)$ are defined
respectively, to be
\begin{align*}
 A  \left( {x_1 ,x_2 , \ldots ,x_n } \right) &= \sum\limits_{k = 1}^n {t_k x_k }  \\
 G  \left( {x_1 ,x_2 , \ldots ,x_n } \right) &= \prod\limits_{k = 1}^n {\left( {x_k } \right)^{t_k } }  \\
 H \left( {x_1 ,x_2 , \ldots ,x_n } \right) &= \frac{1}{{A_{t_k } \left( {\frac{{t_1 }}{{x_1 }},\frac{{t_2 }}{{x_2 }}, \ldots ,\frac{{t_n }}{{x_n }}} \right)}} = \frac{1}{{\sum\limits_{k = 1}^n {\frac{{t_k }}{{x_k }}} }},
 \end{align*}
where $t_k \in \left [0,1\right]$ such that $\sum\limits_{k = 1}^n
{t_k } = 1$ and  $\left( {x_1 ,x_2 , \ldots ,x_n } \right)\in
\left(0,\infty\right)^n$. The weighted form of the HM--GM--AM
inequality is known as (\cite{NP}, p. 11):
\begin{align*}
 H  \left( {x_1 ,x_2 , \ldots ,x_n } \right)\le G \left( {x_1 ,x_2 , \ldots ,x_n } \right)\le
 A
\left( {x_1 ,x_2 , \ldots ,x_n } \right).
\end{align*}
Let $w_1, w_2,\cdots,w_n$ be positive real numbers $(n \ge 2)$ and
$h:J\to \mathbb{R}$ be a non-negative supermultiplicative
function. In \cite{V}, Varo\v{s}anec discussed the case that, if
 $f$ is a non-negative  ${\rm{A_tA_h}}$-convex on $I$, then for $x_1,x_2,\cdots,x_n \in I$ the following inequality
 holds
\begin{align*}
f\left( {\frac{1}{{W_n }}\sum\limits_{k = 1}^n {w_k x_k } }
\right) \le \sum\limits_{k = 1}^n {h\left( {\frac{{w_k }}{{W_{n }
}}} \right) f\left( {x_k } \right)},
\end{align*}
where $W_n=\sum\limits_{k = 1}^n{w_k}$. If $h$ is
submultiplicative function and
 $f$ is an $h$-${\rm{A_tA_t}}$-concave then inequality  is
 reversed. A converse result was also given in \cite{V}.  For more new results
 see \cite{D2}, \cite{D3},  \cite{MNNA}, \cite{OTG}, \cite{PS} and
 \cite{XWQ}.\\

In what follows, Jensen's type inequalities for
${\rm{M_tN_h}}$-convex functions are introduced.
\begin{theorem}
\label{thm11}Let $w_1, w_2,\cdots,w_n$ be positive real numbers
$(n \ge 2)$, and $W_n=\sum\limits_{k = 1}^n{w_k}$.
\begin{enumerate}
\item  If $h$ is a non-negative supermultiplicative function and
 $f$ is a non-negative  ${\rm{A_tG_h}}$-convex on $I$, then for $x_1,x_2,\cdots,x_n \in I$ the following inequality
 holds
\begin{align}
f\left( {\frac{1}{{W_n }}\sum\limits_{k = 1}^n {w_k x_k } }
\right) \le \prod\limits_{k = 1}^n {\left\{ {\left[ {f\left( {x_k
} \right)} \right]^{h\left( {\frac{{w_k }}{{W_{n } }}} \right)} }
\right\}}. \label{eq4.1}
\end{align}
If $h$ is submultiplicative function and
 $f$ is an  ${\rm{A_tG_h}}$-concave then inequality  is
 reversed.\\

\item  If $h$ is a non-negative submultiplicative function and
 $f$ is a non-negative  ${\rm{A_tH_h}}$-convex on $I$, then for $x_1,x_2,\cdots,x_n \in I$ the following inequality
 holds
 \begin{align}
f\left( {\frac{1}{{W_n }}\sum\limits_{k = 1}^n {w_k x_k } }
\right) \le \left( {\sum\limits_{k = 1}^n {\frac{{h\left(
{{\textstyle{{w_k } \over{W_n }}}} \right)}}{{f\left( {x_k }
\right)}}}} \right)^{ - 1}. \label{eq4.2}
\end{align}
If $h$ is supermultiplicative function and
 $f$ is an  ${\rm{A_tH_h}}$-concave then inequality  is reversed.
\end{enumerate}
\end{theorem}

\begin{proof}
Our proof carries by induction. In case $n=2$, the both results
hold.

\begin{enumerate}
\item Assume \eqref{eq4.1} holds for $n-1$ and we are going to
prove it for $n$.
\begin{align*}
 f\left( {\frac{1}{{W_n }}\sum\limits_{k = 1}^n {w_k x_k } } \right)
  &= f\left( {\frac{{w_n }}{{W_n }}x_n  + \sum\limits_{k = 1}^{n - 1} {\frac{{w_k }}{{W_n }}x_k } } \right) \\
  &= f\left( {\frac{{w_n }}{{W_n }}x_n  + \frac{{W_{n - 1} }}{{W_n }}\sum\limits_{k = 1}^{n - 1} {\frac{{w_k }}{{W_{n - 1} }}x_k } } \right) \\
  &\le \left[ {f\left( {x_n } \right)} \right]^{h\left( {\frac{{w_n }}{{W_n }}} \right)} \left[ {f\left( {\sum\limits_{k = 1}^{n - 1} {\frac{{w_k }}{{W_{n - 1} }}x_k } } \right)} \right]^{h\left( {\frac{{W_{n - 1} }}{{W_n }}} \right)}  \\
  &\le \left[ {f\left( {x_n } \right)} \right]^{h\left( {\frac{{w_n }}{{W_n }}} \right)}  \cdot \prod\limits_{k = 1}^{n - 1} {\left\{ {\left[ {f\left( {x_k } \right)} \right]^{h\left( {\frac{{W_{n - 1} }}{{W_n }}} \right)h\left( {\frac{{w_k }}{{W_{n - 1} }}} \right)} } \right\}}  \\
  &= \prod\limits_{k = 1}^n {\left\{ {\left[ {f\left( {x_k } \right)} \right]^{h\left( {\frac{{w_k }}{{W_{n } }}} \right)} } \right\}},  \\
 \end{align*}
and this proves the desired result in \eqref{eq4.1}.

\item  Assume \eqref{eq4.2} holds for $n-1$ and we are going to
prove it for $n$.
\begin{align*}
  f\left( {\frac{1}{{W_n }}\sum\limits_{k = 1}^n {w_k x_k } } \right)
   = f\left( {\frac{{w_n }}{{W_n }}x_n  + \sum\limits_{k = 1}^{n - 1} {\frac{{w_k }}{{W_n }}x_k } }
  \right)
  &= f\left( {\frac{{w_n }}{{W_n }}x_n  + \frac{{W_{n - 1} }}{{W_n }}\sum\limits_{k = 1}^{n - 1} {\frac{{w_k }}{{W_{n - 1} }}x_k } } \right) \\
  &\le \frac{1}{{   \frac{h\left( {\frac{{w_n }}{{W_n }}} \right)}{f\left( {x_n} \right)} + \frac{h\left( {\frac{{W_{n - 1} }}{{W_n }}} \right)}{f\left( {\sum\limits_{k = 1}^{n - 1} {\frac{{w_k }}{{W_{n - 1} }}x_k } } \right)}}} \\
  &\le \frac{1}{{   \frac{h\left( {\frac{{w_n }}{{W_n }}} \right)}{f\left( {x_n} \right)} +h\left( {{\textstyle{{W_{n - 1} } \over {W_n  }}}}\right)\sum\limits_{k = 1}^{n - 1}{\frac{{h\left({{\textstyle{{w_k } \over {W_{n-1} }}}}\right)}}{{f\left( {x_k }\right)}}}}} \\
    &\le \frac{1}{{   \frac{h\left( {\frac{{w_n }}{{W_n }}} \right)}{f\left( {x_n} \right)} +\sum\limits_{k = 1}^{n - 1} {\frac{{h\left( {{\textstyle{{w_k }\over {W_n }}}} \right)}}{{f\left( {x_k } \right)}}}}}
   \le \frac{1}{\sum\limits_{k = 1}^n {\frac{{h\left( {{\textstyle{{w_k } \over{W_n }}}} \right)}}{{f\left( {x_k }
   \right)}}}},
 \end{align*}
which proves the desired result in \eqref{eq4.2}.
\end{enumerate}
Hence, by Mathematical Induction both statements are hold for all
$n\ge2$, and therefore the proof is completely established.
\end{proof}

The corresponding converse versions of Jensen inequality for
 ${\rm{A_tG_h}}$-convex and ${\rm{A_tH_h}}$-convex are
incorporated in the following theorem.
\begin{theorem}
Let $w_1, w_2,\cdots,w_n$ be positive real numbers $(n \ge 2)$,
and $\left(m,M\right)\subseteq I$.
\begin{enumerate}
\item If $h:\left(0,\infty\right) \to \left(0,\infty\right)$ is a
non-negative supermultiplicative function and
 $f$ is positive  ${\rm{A_tG_h}}$-convex, then for every finite sequence of points $x_1,\cdots, x_n \in \left(m,M\right)\subseteq I$
 we have
 \begin{align}
 \prod\limits_{k = 1}^n {\left[f\left( x_k \right)\right]^{h\left(
{\frac{{w_k }}{{W_n }}} \right)} } \le   \prod\limits_{k = 1}^n
{\left\{\left[ {f\left( m  \right)} \right]^{h\left(\frac{{M -
x_k}}{{M  - m }}\cdot \frac{{w_k }}{{W_n }}\right)\cdot  } \left[
{f\left( M \right)} \right]^{h\left(\frac{{ x_k - m }}{{ M - m
}}\cdot \frac{{w_k }}{{W_n }}\right) }\right\}}, \label{eq4.3}
\end{align}
If $h$ is submultiplicative function and $f$ is an
 ${\rm{A_tG_h}}$-concave then inequality  is reversed.\\

\item If $h:\left(0,\infty\right) \to \left(0,\infty\right)$ is a
non-negative submultiplicative function and
 $f$ is positive  ${\rm{A_tH_h}}$-convex, then for every finite sequence of points $x_1,\cdots, x_n \in \left(m,M\right) \subseteq I$
 we have
 \begin{align}
\left(\sum\limits_{k = 1}^n  \frac{h\left( {\frac{{w_k }}{{W_n }}}
\right) }{f\left( x_k \right)} \right)^{-1}\le
\left(\sum\limits_{k = 1}^n \frac{ h\left( {\frac{{x_k  - m}}{{M -
m}}} \right) f\left( m \right) + h\left( {\frac{{M-x_k}}{{M - m}}}
\right) f\left(M \right)}{ f\left( m \right)f\left( M
\right)}h\left( {\frac{{w_k }}{{W_n }}} \right) \right)^{-1},
\label{eq4.4}
\end{align}
If $h$ is supermultiplicative function and $f$ is an
 ${\rm{A_tH_h}}$-concave then inequality  is reversed.\\
\end{enumerate}

\end{theorem}

\begin{proof}

\begin{enumerate}
\item In \eqref{eqhAG}, setting $m=x_1$, $x_2=x_k$ and $x_3=M$ we
get
\begin{align*}
f\left( x_k \right) \le \left[ {f\left( m  \right)}
\right]^{h\left(\frac{{M  - x_k}}{{M  - m }}\right)} \left[
{f\left( M \right)} \right]^{h\left(\frac{{ x_k - m }}{{ M - m
}}\right)}.
\end{align*}
Since $f$ is positive therefore we have
\begin{align*}
\left[f\left( x_k \right)\right]^{h\left( {\frac{{w_k }}{{W_n }}}
\right)} &\le  \left[ {f\left( m  \right)}
\right]^{h\left(\frac{{M - x_k}}{{M  - m }}\right)\cdot h\left(
{\frac{{w_k }}{{W_n }}} \right)} \left[ {f\left( M \right)}
\right]^{h\left(\frac{{ x_k - m }}{{ M - m }}\right)\cdot h\left(
{\frac{{w_k }}{{W_n }}} \right)}
\\
&\le \left[ {f\left( m  \right)} \right]^{h\left(\frac{{M -
x_k}}{{M  - m }}\cdot \frac{{w_k }}{{W_n }}\right)\cdot  } \left[
{f\left( M \right)} \right]^{h\left(\frac{{ x_k - m }}{{ M - m
}}\cdot \frac{{w_k }}{{W_n }}\right) },
\end{align*}
Multiplying the above inequality up to $n$ we get the required
results in \eqref{eq4.3}.\\

\item   Setting $m=x_1$, $x_2=x_k$ and $x_3=M$ in the reverse of
\eqref{eqhAH} we get
\begin{align*}
f\left( x_k \right) \le \frac{{  f\left( m \right)f\left( M
\right)}}{{h\left( {\frac{{x_k  - m}}{{M - m}}} \right)f\left( m
\right) + h\left( {\frac{{M-x_k}}{{M - m}}} \right)f\left(M
\right)}}.
\end{align*}
  Reversing  the inequality and then multiplying the above inequality  by $h\left( {\frac{{w_k
}}{{W_n }}} \right) $ we get
\begin{align*}
\frac{h\left( {\frac{{w_k }}{{W_n }}} \right) }{f\left( x_k
\right)} \ge \frac{h\left( {\frac{{x_k  - m}}{{M - m}}}
\right)f\left( m \right) + h\left( {\frac{{M-x_k}}{{M - m}}}
\right)f\left(M \right)}{ f\left( m \right)f\left( M
\right)}h\left( {\frac{{w_k }}{{W_n }}} \right).
\end{align*}
Summing up to $n$ and then reverse  the above inequality, we get
the required result in \eqref{eq4.4}.
\end{enumerate}
\end{proof}

\begin{theorem}
Let $w_1, w_2,\cdots,w_n$ be positive real numbers $(n \ge 2)$,
and $W_n=\sum\limits_{k = 1}^n{w_k}$.
\begin{enumerate}
\item If $h$ is a non-negative supermultiplicative function and
 $f$ is positive  ${\rm{G_tA_h}}$-convex on $I$, then for $x_1,x_2,\cdots,x_n \in I$ the following inequality
 holds
 \begin{align}
 f\left( {\prod\limits_{k = 1}^n {\left( {x_k } \right)^{\frac{{w_k }}{{W_n }}} } }
 \right)\le\sum\limits_{k = 1}^n {h\left( {\frac{{w_k }}{{W_n }}} \right)f\left( {x_k } \right)}. \label{eq4.5}
\end{align}
If $h$ is submultiplicative function and $f$ is an
 ${\rm{G_tA_h}}$-concave then inequality  is reversed.\\

\item If $h$ is a non-negative supermultiplicative function and
 $f$ is positive  ${\rm{G_tG_h}}$-convex on $I$, then for $x_1,x_2,\cdots,x_n \in I$ the following inequality
 holds
 \begin{align}
f\left( {\prod\limits_{k = 1}^n {\left( {x_k } \right)^{\frac{{w_k
}}{{W_n }}} } } \right) \le\prod\limits_{k = 1}^n {\left[ {f\left(
{x_k } \right)} \right]^{h\left( {\frac{{w_k }}{{W_n }}} \right)}
}. \label{eq4.6}
\end{align}
If $h$ is submultiplicative function and $f$ is an
 ${\rm{G_tG_h}}$-concave then inequality  is reversed.\\

\item If $h$ is a non-negative submultiplicative function and
 $f$ is  positive  ${\rm{G_tH_h}}$-convex on $I$, then for $x_1,x_2,\cdots,x_n \in I$ the following inequality
 holds
 \begin{align}
f\left( {\prod\limits_{k = 1}^n {\left( {x_k } \right)^{\frac{{w_k
}}{{W_n }}} } } \right) \le \frac{1}{\sum\limits_{k = 1}^n
{\frac{{h\left( {{\textstyle{{w_k } \over{W_n }}}}
\right)}}{{f\left( {x_k } \right)}}}}. \label{eq4.7}
\end{align}
If $h$ is supermultiplicative function and $f$ is an
 ${\rm{G_tH_h}}$-concave then inequality  is reversed.
\end{enumerate}
\end{theorem}
\begin{proof}
Our proof carries by induction. In case $n=2$, the  results hold
by definition.
\begin{enumerate}
\item Assume \eqref{eq4.5} holds for $n-1$ and we are going to
prove it for $n$.
\begin{align*}
  f\left( {\prod\limits_{k = 1}^n {\left( {x_k } \right)^{\frac{{w_k }}{{W_n }}} } } \right) &= f\left( {\left( {x_n } \right)^{\frac{{w_n }}{{W_n }}}  \cdot \prod\limits_{k = 1}^{n - 1} {\left( {x_k } \right)^{\frac{{w_k }}{{W_n }}} } } \right) \\
  &= f\left( {\left( {x_n } \right)^{\frac{{w_n }}{{W_n }}}  \cdot \prod\limits_{k = 1}^{n - 1} {\left( {x_k } \right)^{\frac{{w_k }}{{W_{n - 1} }}\frac{{W_{n - 1} }}{{W_n }}} } } \right) \\
  &\le h\left( {\frac{{w_n }}{{W_n }}} \right)f\left( {x_n } \right) + h\left( {\frac{{W_{n - 1} }}{{W_n }}} \right)f\left( {\sum\limits_{k = 1}^{n - 1} {\frac{{w_k }}{{W_{n - 1} }}x_k } } \right) \\
  &\le h\left( {\frac{{w_n }}{{W_n }}} \right)f\left( {x_n } \right) + h\left( {\frac{{W_{n - 1} }}{{W_n }}} \right)\sum\limits_{k = 1}^{n - 1} {h\left( {\frac{{w_k }}{{W_{n - 1} }}} \right)f\left( {x_k } \right)}  \\
  &\le h\left( {\frac{{w_n }}{{W_n }}} \right)f\left( {x_n } \right) + \sum\limits_{k = 1}^{n - 1} {h\left( {\frac{{w_k }}{{W_n }}} \right)f\left( {x_k } \right)}
  = \sum\limits_{k = 1}^n {h\left( {\frac{{w_k }}{{W_n }}} \right)f\left( {x_k }
  \right)},
 \end{align*}
which proves the desired result in \eqref{eq4.5}.

 \item Assume \eqref{eq4.6} holds for $n-1$ and we are going to prove
it for $n$.
\begin{align*}
 f\left( {\prod\limits_{k = 1}^n {\left( {x_k } \right)^{\frac{{w_k }}{{W_n }}} } } \right)
  &\le \left[ {f\left( {x_n } \right)} \right]^{h\left( {\frac{{w_n }}{{W_n }}} \right)} \left[ {f\left( {\prod\limits_{k = 1}^{n - 1} {\frac{{w_k }}{{W_{n - 1} }}x_k } } \right)} \right]^{h\left( {\frac{{W_{n - 1} }}{{W_n }}} \right)}  \\
  &\le \left[ {f\left( {x_n } \right)} \right]^{h\left( {\frac{{w_n }}{{W_n }}} \right)} \left[ {\prod\limits_{k = 1}^{n - 1} {\left( {f\left( {x_k } \right)} \right)^{h\left( {\frac{{w_k }}{{W_{n - 1} }}} \right)} } } \right]^{h\left( {\frac{{W_{n - 1} }}{{W_n }}} \right)}  \\
  &\le \left[ {f\left( {x_n } \right)} \right]^{h\left( {\frac{{w_n }}{{W_n }}} \right)} \left[ {\prod\limits_{k = 1}^{n - 1} {\left( {f\left( {x_k } \right)} \right)^{h\left( {\frac{{w_k }}{{W_{n - 1} }}} \right)h\left( {\frac{{W_{n - 1} }}{{W_n }}} \right)} } } \right] \\
  &\le \left[ {f\left( {x_n } \right)} \right]^{h\left( {\frac{{w_n }}{{W_n }}} \right)} \left[ {\prod\limits_{k = 1}^{n - 1} {\left( {f\left( {x_k } \right)} \right)^{h\left( {\frac{{w_k }}{{W_n }}} \right)} } } \right]
   = \prod\limits_{k = 1}^n {\left[ {f\left( {x_k } \right)} \right]^{h\left( {\frac{{w_k }}{{W_n }}}
   \right)}},
 \end{align*}
which proves the desired result in \eqref{eq4.6}.

\item Assume \eqref{eq4.7} holds for $n-1$ and we are going to
prove it for $n$.
\begin{align*}
 f\left( {\prod\limits_{k = 1}^n {\left( {x_k } \right)^{\frac{{w_k }}{{W_n }}} } } \right)
   &\le \frac{1}{{   \frac{h\left( {\frac{{w_n }}{{W_n }}} \right)}{f\left( {x_n} \right)} + \frac{h\left( {\frac{{W_{n - 1} }}{{W_n }}} \right)}{f\left( {\sum\limits_{k = 1}^{n - 1} {\frac{{w_k }}{{W_{n - 1} }}x_k } } \right)}}} \\
  &\le \frac{1}{{   \frac{h\left( {\frac{{w_n }}{{W_n }}} \right)}{f\left( {x_n} \right)} +h\left( {{\textstyle{{W_{n - 1} } \over {W_n  }}}}\right)\sum\limits_{k = 1}^{n - 1}{\frac{{h\left({{\textstyle{{w_k } \over {W_{n-1} }}}}\right)}}{{f\left( {x_k }\right)}}}}} \\
    &\le \frac{1}{{   \frac{h\left( {\frac{{w_n }}{{W_n }}} \right)}{f\left( {x_n} \right)} +\sum\limits_{k = 1}^{n - 1} {\frac{{h\left( {{\textstyle{{w_k }\over {W_n }}}} \right)}}{{f\left( {x_k } \right)}}}}}
   \le \frac{1}{\sum\limits_{k = 1}^n {\frac{{h\left( {{\textstyle{{w_k } \over{W_n }}}} \right)}}{{f\left( {x_k } \right)}}}},
 \end{align*}
 which proves the desired result in \eqref{eq4.7}.
\end{enumerate} Hence, by Mathematical Induction both statements
are hold for all $n\ge2$, and therefore the proof is completely
established.
\end{proof}

The corresponding converse versions of Jensen inequality for
 ${\rm{G_tA_h}}$-convex,  ${\rm{G_tG_h}}$-convex and
 ${\rm{G_tH_h}}$-convex are incorporated in the following
theorem.
\begin{theorem}
Let $w_1, w_2,\cdots,w_n$ be positive real numbers $(n \ge 2)$,
and $\left(m,M\right)\subseteq I$.
\begin{enumerate}
\item If $h:\left(m,M\right) \to \left[m,M\right)$ is a
non-negative supermultiplicative function and
 $f$ is positive ${\rm{G_tA_h}}$-convex, then for every finite sequence of points $x_1,\cdots, x_n \in \left(m,M\right)$ $(x_k<x_{k+ 1})$
 we have
 \begin{multline}
\sum\limits_{k = 1}^n {h\left( {\frac{{w_k }}{{W_n }}} \right)
f\left( x_k \right)}
\\
\le     \sum\limits_{k = 1}^n { h\left( {\frac{{w_k }}{{W_n }}}
\right)\cdot\left[h\left( \frac{{\ln \left(M \right) - \ln
\left(x_k \right) }}{{ \ln \left(M  \right) - \ln \left(m \right)
}} \right)\cdot f\left( m \right) + h\left(\frac{{\ln \left( x_k
\right) - \ln \left( m \right) }}{{ \ln \left( M  \right) - \ln
\left(m \right)}}\right) \cdot f\left(M \right)\right]
}.\label{eq4.8}
\end{multline}
If $h$ is submultiplicative function and $f$ is an
${\rm{G_tA_h}}$-concave then inequality  is reversed.\\

\item If $h:\left(0,\infty\right) \to \left(0,\infty\right)$ is a
non-negative supermultiplicative function and
 $f$ is positive ${\rm{G_tG_h}}$-convex, then for every finite sequence of points $x_1,\cdots, x_n \in \left(m,M\right)\subseteq I$
 we have
 \begin{align}
\prod\limits_{k = 1}^n {\left[f\left( x_k \right) \right]^{h\left(
{\frac{{w_k }}{{W_n }}} \right)}}  \le \prod\limits_{k = 1}^n
{\left\{\left[ {f\left( m  \right)} \right]^{h\left( \frac{{\ln
\left(M \right) - \ln \left(x_k \right) }}{{ \ln \left(M  \right)
- \ln \left(m \right) }} \right)\cdot h\left( {\frac{{w_k }}{{W_n
}}} \right)} \left[ {f\left( M \right)} \right]^{h\left(\frac{{\ln
\left( x_k \right) - \ln \left( m \right) }}{{ \ln \left( M
\right) - \ln \left(m \right)}}\right)\cdot h\left( {\frac{{w_k
}}{{W_n }}} \right)}\right\}}. \label{eq4.9}
\end{align}
If $h$ is submultiplicative function and $f$ is an
${\rm{G_tG_h}}$-concave then inequality  is reversed.\\

\item If $h:\left(0,\infty\right) \to \left(0,\infty\right)$ is a
non-negative submultiplicative function and
 $f$ is positive ${\rm{G_tH_h}}$-convex, then for every finite sequence of points $x_1,\cdots, x_n \in \left(m,M\right)\subseteq I$
 we have
 \begin{align}
\left(\sum\limits_{k = 1}^n{\frac{h\left( {\frac{{w_k }}{{W_n }}}
\right)}{f\left( x_k \right)} } \right)^{-1} \le
\left(\sum\limits_{k = 1}^n {\frac{ h\left(\frac{\ln x_k - \ln
m}{\ln M - \ln m}\right)f\left(m \right) + h\left(\frac{\ln M -
\ln x_k}{\ln M - \ln m}\right) f\left( M \right)}{f\left( m
\right)f\left( M \right) }h\left( {\frac{{w_k }}{{W_n }}} \right)}
\right)^{-1}.\label{eq4.10}
\end{align}
If $h$ is supermultiplicative function and $f$ is an
${\rm{G_tH_h}}$-concave then inequality  is reversed.\\
\end{enumerate}

\end{theorem}

\begin{proof}

\begin{enumerate}
\item In \eqref{eqhGA}, setting $m=x_1$, $x_2=x_k$ and $x_3=M$ we
get
\begin{align*}
f\left( x_k \right)\le h\left( \frac{{\ln \left(M \right) - \ln
\left(x_k \right) }}{{ \ln \left(M  \right) - \ln \left(m \right)
}} \right)\cdot f\left( m \right) + h\left(\frac{{\ln \left( x_k
\right) - \ln \left( m \right) }}{{ \ln \left( M  \right) - \ln
\left(m \right)}}\right) \cdot f\left(M \right)
\end{align*}
Multiplying the above inequality  by $h\left( {\frac{{w_k }}{{W_n
}}} \right) $ and summing up to $n$ we get the required results in
\eqref{eq4.8}.

\item  Setting $m=x_1$, $x_2=x_k$ and $x_3=M$ in \eqref{eqhGG} we
get
\begin{align*}
f\left( x_k \right)  \le \left[ {f\left( m  \right)}
\right]^{h\left( \frac{{\ln \left(M \right) - \ln \left(x_k
\right) }}{{ \ln \left(M  \right) - \ln \left(m \right) }}
\right)} \left[ {f\left( M \right)} \right]^{h\left(\frac{{\ln
\left( x_k \right) - \ln \left( m \right) }}{{ \ln \left( M
\right) - \ln \left(m \right)}}\right)}.
\end{align*}
Since $f$ is positive, the above inequality implies that
\begin{align*}
\left[f\left( x_k \right) \right]^{h\left( {\frac{{w_k }}{{W_n }}}
\right)} \le \left[ {f\left( m  \right)} \right]^{h\left(
\frac{{\ln \left(M \right) - \ln \left(x_k \right) }}{{ \ln
\left(M  \right) - \ln \left(m \right) }} \right)\cdot h\left(
{\frac{{w_k }}{{W_n }}} \right)} \left[ {f\left( M \right)}
\right]^{h\left(\frac{{\ln \left( x_k \right) - \ln \left( m
\right) }}{{ \ln \left( M \right) - \ln \left(m
\right)}}\right)\cdot h\left( {\frac{{w_k }}{{W_n }}} \right)}.
\end{align*}
Multiplying the above inequality  up to $n$ we get the required
result in \eqref{eq4.9}.

\item  Since $f$ is  ${\rm{G_tH_h}}$-convex, then \eqref{eqhGH}
holds.
\begin{align*}
f\left( x_k \right) &\le \frac{{f\left( m \right)f\left( M \right)
}}{{ h\left(\frac{\ln x_k - \ln m}{\ln M - \ln m}\right)f\left(m
\right) + h\left(\frac{\ln M - \ln x_k}{\ln M - \ln m}\right)
f\left( M \right)}}.
\end{align*}
Reversing the order in the inequality we get
\begin{align*}
\frac{1}{f\left( x_k \right)} &\ge \frac{ h\left(\frac{\ln x_k -
\ln m}{\ln M - \ln m}\right)f\left(m \right) + h\left(\frac{\ln M
- \ln x_k}{\ln M - \ln m}\right) f\left( M \right)}{f\left( m
\right)f\left( M \right) }.
\end{align*}
Multiplying both sides by $h\left( {\frac{{w_k }}{{W_n }}}
\right)$ and summing up to $n$ we get
\begin{align*}
\sum\limits_{k = 1}^n{\frac{h\left( {\frac{{w_k }}{{W_n }}}
\right)}{f\left( x_k \right)} }&\ge \sum\limits_{k = 1}^n {\frac{
h\left(\frac{\ln x_k - \ln m}{\ln M - \ln m}\right)f\left(m
\right) + h\left(\frac{\ln M - \ln x_k}{\ln M - \ln m}\right)
f\left( M \right)}{f\left( m \right)f\left( M \right) }h\left(
{\frac{{w_k }}{{W_n }}} \right)}.
\end{align*}
Reversing the order in the inequality again we get the required
result in \eqref{eq4.10}.
\end{enumerate}

\end{proof}

\begin{theorem}
Let $w_1, w_2,\cdots,w_n$ be positive real numbers $(n \ge 2)$,
and $W_n=\sum\limits_{k = 1}^n{w_k}$.
\begin{enumerate}
\item If $h$ is a non-negative supermultiplicative function and
 $f$ is positive  ${\rm{H_tA_h}}$-convex on $I$, then for $x_1,x_2,\cdots,x_n \in I$ the following inequality
 holds
 \begin{align}
f\left( { \left( \frac{1}{{W_n }}\sum\limits_{k = 1}^n {\frac{{w_k
}}{{x_k }}}  \right)^{-1}}
  \right) \le \sum\limits_{k = 1}^n {h\left(
{\frac{{w_k }}{{W_n }}} \right)f\left( {x_k } \right)}.
\label{eq4.11}
\end{align}
If $h$ is submultiplicative function and $f$ is an
${\rm{H_tA_h}}$-concave then inequality  is reversed.\\

\item If $h$ is a non-negative supermultiplicative function and
 $f$ is positive ${\rm{H_tG_h}}$-convex on $I$, then for $x_1,x_2,\cdots,x_n \in I$ the following inequality
 holds
 \begin{align}
f\left( { \left( \frac{1}{{W_n }}\sum\limits_{k = 1}^n {\frac{{w_k
}}{{x_k }}}  \right)^{-1}}
  \right)\le\prod\limits_{k = 1}^n {\left[
{f\left( {x_k } \right)} \right]^{h\left( {\frac{{w_k }}{{W_n }}}
\right)} }.\label{eq4.12}
\end{align}
If $h$ is submultiplicative function and $f$ is an
 ${\rm{H_tG_h}}$-concave then inequality  is reversed.\\

\item If $h$ is a non-negative submultiplicative function and
 $f$ is positive  ${\rm{H_tH_h}}$-convex on $I$, then for $x_1,x_2,\cdots,x_n \in I$ the following inequality
 holds
 \begin{align}
f\left( { \left( \frac{1}{{W_n }}\sum\limits_{k = 1}^n {\frac{{w_k
}}{{x_k }}}  \right)^{-1}}
  \right)\le \left( {\sum\limits_{k
= 1}^n {\frac{h\left( {\frac{{w_k }}{{W_n }}} \right)}{f\left(
{x_k} \right)}} } \right)^{ - 1}.\label{eq4.13}
\end{align}
If $h$ is supermultiplicative function and $f$ is an
 ${\rm{H_tH_h}}$-concave then inequality  is reversed.
\end{enumerate}
\end{theorem}

\begin{proof}
Our proof carries by induction. In case $n=2$, both results hold.
\begin{enumerate}

\item Assume \eqref{eqhHA} holds for $n-1$ and we are going to
prove it for $n$.
\begin{align*}
 f\left( {\frac{1}{{\sum\limits_{k = 1}^n {\frac{{w_k }}{{W_n }}\frac{1}{{x_k }}} }}} \right) &= f\left( {\frac{1}{{\frac{{w_n }}{{W_n }}\frac{1}{{x_n }} + \sum\limits_{k = 1}^{n - 1} {\frac{{w_k }}{{W_n }}\frac{1}{{x_k }}} }}} \right) \\
  &= f\left( {\frac{1}{{\frac{{w_n }}{{W_n }}\frac{1}{{x_n }} + \frac{{W_{n - 1} }}{{W_n }}\sum\limits_{k = 1}^{n - 1} {\frac{{w_k }}{{W_{n - 1} }}\frac{1}{{x_k }}} }}} \right) \\
  &\le h\left( {\frac{{w_n }}{{W_n }}} \right)f\left( {x_n } \right) + h\left( {\frac{{W_{n - 1} }}{{W_n }}} \right)f\left( {\sum\limits_{k = 1}^{n - 1} {\frac{{w_k }}{{W_{n - 1} }}x_k } } \right) \\
  &\le h\left( {\frac{{w_n }}{{W_n }}} \right)f\left( {x_n } \right) + h\left( {\frac{{W_{n - 1} }}{{W_n }}} \right)\sum\limits_{k = 1}^{n - 1} {h\left( {\frac{{w_k }}{{W_{n - 1} }}} \right)f\left( {x_k } \right)}  \\
  &\le h\left( {\frac{{w_n }}{{W_n }}} \right)f\left( {x_n } \right) + \sum\limits_{k = 1}^{n - 1} {h\left( {\frac{{w_k }}{{W_n }}} \right)f\left( {x_k } \right)}  \\
  &= \sum\limits_{k = 1}^n {h\left( {\frac{{w_k }}{{W_n }}} \right)f\left( {x_k } \right)},  \\
 \end{align*}
which proves the desired result in \eqref{eq4.11}.

  \item Assume \eqref{eqhHG} holds for $n-1$ and we are going to
prove it for $n$.
\begin{align*}
 f\left( {\frac{1}{{\sum\limits_{k = 1}^n {\frac{{w_k }}{{W_n }}\frac{1}{{x_k }}} }}} \right)
   &\le \left[ {f\left( {x_n } \right)} \right]^{h\left( {\frac{{w_n }}{{W_n }}} \right)} \left[ {f\left( {\prod\limits_{k = 1}^{n - 1} {\frac{{w_k }}{{W_{n - 1} }}x_k } } \right)} \right]^{h\left( {\frac{{W_{n - 1} }}{{W_n }}} \right)}  \\
  &\le \left[ {f\left( {x_n } \right)} \right]^{h\left( {\frac{{w_n }}{{W_n }}} \right)} \left[ {\prod\limits_{k = 1}^{n - 1} {\left( {f\left( {x_k } \right)} \right)^{h\left( {\frac{{w_k }}{{W_{n - 1} }}} \right)} } } \right]^{h\left( {\frac{{W_{n - 1} }}{{W_n }}} \right)}  \\
  &\le \left[ {f\left( {x_n } \right)} \right]^{h\left( {\frac{{w_n }}{{W_n }}} \right)} \left[ {\prod\limits_{k = 1}^{n - 1} {\left( {f\left( {x_k } \right)} \right)^{h\left( {\frac{{w_k }}{{W_{n - 1} }}} \right)h\left( {\frac{{W_{n - 1} }}{{W_n }}} \right)} } } \right] \\
  &\le \left[ {f\left( {x_n } \right)} \right]^{h\left( {\frac{{w_n }}{{W_n }}} \right)} \left[ {\prod\limits_{k = 1}^{n - 1} {\left( {f\left( {x_k } \right)} \right)^{h\left( {\frac{{w_k }}{{W_n }}} \right)} } } \right]
  = \prod\limits_{k = 1}^n {\left[ {f\left( {x_k } \right)} \right]^{h\left( {\frac{{w_k }}{{W_n }}} \right)}
  },
 \end{align*}
which proves the desired result in \eqref{eq4.12}.

\item Assume \eqref{eqhHH} holds for $n-1$ and we are going to
prove it for $n$.
\begin{align*}
  f\left( {\frac{1}{{\frac{1}{{W_n }}\sum\limits_{k = 1}^n {\frac{{w_k }}{{x_k }}} }}}
  \right)
  &\le \frac{1}{{   \frac{h\left( {\frac{{w_n }}{{W_n }}} \right)}{f\left( {x_n} \right)} + \frac{h\left( {\frac{{W_{n - 1} }}{{W_n }}} \right)}{f\left( {\sum\limits_{k = 1}^{n - 1} {\frac{{w_k }}{{W_{n - 1} }}x_k } } \right)}}} \\
  &\le \frac{1}{{   \frac{h\left( {\frac{{w_n }}{{W_n }}} \right)}{f\left( {x_n} \right)} +h\left( {{\textstyle{{W_{n - 1} } \over {W_n  }}}}\right)\sum\limits_{k = 1}^{n - 1}{\frac{{h\left({{\textstyle{{w_k } \over {W_{n-1} }}}}\right)}}{{f\left( {x_k }\right)}}}}} \\
    &\le \frac{1}{{   \frac{h\left( {\frac{{w_n }}{{W_n }}} \right)}{f\left( {x_n} \right)} +\sum\limits_{k = 1}^{n - 1} {\frac{{h\left( {{\textstyle{{w_k }\over {W_n }}}} \right)}}{{f\left( {x_k } \right)}}}}}
   \le \frac{1}{\sum\limits_{k = 1}^n {\frac{{h\left( {{\textstyle{{w_k } \over{W_n }}}} \right)}}{{f\left( {x_k } \right)}}}},
 \end{align*}
 which proves the desired result in \eqref{eq4.13}.
 \end{enumerate}
Hence, by Mathematical Induction the three statements are hold for
all $n\ge2$, and therefore the proof is completely established.
\end{proof}

The corresponding converse versions of Jensen inequality for
 ${\rm{H_tA_h}}$-convex,  ${\rm{H_tG_h}}$-convex and
 ${\rm{H_tH_h}}$-convex are incorporated in the following
theorem.
\begin{theorem}
Let $w_1, w_2,\cdots,w_n$ be positive real numbers $(n \ge 2)$,
and $\left(m,M\right)\subseteq I$.
\begin{enumerate}
\item If $h:\left(0,\infty\right) \to \left(0,\infty\right)$ is a
non-negative supermultiplicative function and
 $f$ is positive ${\rm{H_tA_h}}$-convex, then for every finite sequence of points $x_1,\cdots, x_n \in \left(m,M\right)\subseteq I$
 we have
 \begin{align}
\sum\limits_{k = 1}^n {h\left( {\frac{{w_k }}{{W_n }}} \right)
f\left( x_k \right)} \le \sum\limits_{k = 1}^n {\left[ h\left(
\frac{{m \left(M  - x_k\right) }}{{ x_k\left( {M - m}\right)
}}\right)f\left( m \right) + h\left(\frac{{M\left( {x_k - m }
\right) }}{{ x_k\left( {M - m} \right)}}\right) f\left( M \right)
\right]h\left( {\frac{{w_k }}{{W_n }}} \right)}. \label{eq4.14}
\end{align}
If $h$ is submultiplicative function and $f$ is an
 ${\rm{H_tA_h}}$-concave then inequality  is reversed.\\

\item If $h:\left(0,\infty\right) \to \left(0,\infty\right)$ is a
non-negative supermultiplicative function and
 $f$ is positive ${\rm{H_tG_h}}$-convex, then for every finite sequence of points $x_1,\cdots, x_n \in
 \left(m,M\right) \subseteq I$
 we have
 \begin{align}
\prod\limits_{k = 1}^n {\left[f\left( x_k \right) \right]^{h\left(
{\frac{{w_k }}{{W_n }}} \right)}}  \le \prod\limits_{k = 1}^n
{\left\{\left[ {f\left( m  \right)} \right]^{ h\left( \frac{{m
\left(M  - x_k\right) }}{{ x_k\left( {M - m}\right) }}\right)\cdot
h\left( {\frac{{w_k }}{{W_n }}} \right)} \left[ {f\left( M
\right)} \right]^{ h\left(\frac{{M\left( {x_k - m } \right) }}{{
x_k\left( {M - m} \right)}}\right)\cdot h\left( {\frac{{w_k
}}{{W_n }}} \right)}\right\}}. \label{eq4.15}
\end{align}
If $h$ is submultiplicative function and $f$ is an
 ${\rm{H_tG_h}}$-concave then inequality  is reversed.\\

\item If $h:\left(0,\infty\right) \to \left(0,\infty\right)$ is a
non-negative submultiplicative function and
 $f$ is positive  ${\rm{H_tH_h}}$-convex, then for every finite sequence of points $x_1,\cdots, x_n \in \left(m,M\right)\subseteq I$
 we have
 \begin{align}
\left(\sum\limits_{k = 1}^n{\frac{h\left( {\frac{{w_k }}{{W_n }}}
\right)}{f\left( x_k \right)} } \right)^{-1}  \le
\left(\sum\limits_{k = 1}^n {\frac{ h\left(\frac{{M\left( {x_k - m
} \right) }}{{ x_k\left( {M - m} \right)}}\right)f\left(m \right)
+ h\left( \frac{{m \left(M - x_k\right) }}{{ x_k\left( {M -
m}\right) }}\right) f\left( M \right)}{f\left( m \right)f\left( M
\right) }h\left( {\frac{{w_k }}{{W_n }}} \right)} \right)^{-1}.
\label{eq4.16}
\end{align}
If $h$ is supermultiplicative function and $f$ is an ${\rm{H_tH_h}}$-concave then inequality  is reversed.\\
\end{enumerate}

\end{theorem}

\begin{proof}

\begin{enumerate}
\item In \eqref{eqhHA}, setting $m=x_1$, $x_2=x_k$ and $x_3=M$ we
get
\begin{align*}
f\left( x_k \right)  \le  h\left( \frac{{m \left(M  - x_k\right)
}}{{ x_k\left( {M - m}\right)  }}\right)f\left( m \right) +
h\left(\frac{{M\left( {x_k - m } \right) }}{{ x_k\left( {M - m}
\right)}}\right) f\left( M \right)
\end{align*}
Multiplying the above inequality  by $h\left( {\frac{{w_k }}{{W_n
}}} \right) $ and summing up to $n$ we get the required results in
\eqref{eq4.14}.

\item  Setting $m=x_1$, $x_2=x_k$ and $x_3=M$ in \eqref{eqhHG} we
get
\begin{align*}
f\left( x_k \right)  \le  \left[ {f\left( m  \right)} \right]^{
h\left( \frac{{m \left(M  - x_k\right) }}{{ x_k\left( {M -
m}\right)  }}\right)} \left[ {f\left( M  \right)} \right]^{
h\left(\frac{{M\left( {x_k - m } \right) }}{{ x_k\left( {M - m}
\right)}}\right)}.
\end{align*}
Since $f$ is positive, the above inequality implies that
\begin{align*}
\left[f\left( x_k \right) \right]^{h\left( {\frac{{w_k }}{{W_n }}}
\right)} \le\left[ {f\left( m  \right)} \right]^{ h\left( \frac{{m
\left(M  - x_k\right) }}{{ x_k\left( {M - m}\right) }}\right)\cdot
h\left( {\frac{{w_k }}{{W_n }}} \right)} \left[ {f\left( M
\right)} \right]^{ h\left(\frac{{M\left( {x_k - m } \right) }}{{
x_k\left( {M - m} \right)}}\right)\cdot h\left( {\frac{{w_k
}}{{W_n }}} \right)}.
\end{align*}
Multiplying the above inequality  up to $n$ we get the required
result in \eqref{eq4.15}.

\item  Setting $m=x_1$, $x_2=x_k$ and $x_3=M$ in \eqref{eqhHH} we
get
\begin{align*}
f\left( x_k \right) &\le  \frac{{  f\left( x_1  \right)f\left( x_3
\right)}}{{h\left(\frac{{M\left( {x_k - m } \right) }}{{ x_k\left(
{M - m} \right)}}\right)f\left( x_1 \right) + h\left( \frac{{m
\left(M  - x_k\right) }}{{ x_k\left( {M - m}\right)
}}\right)f\left( x_3 \right)}}.
\end{align*}
Reversing the order in the inequality we get
\begin{align*}
\frac{1}{f\left( x_k \right)} &\ge \frac{h\left(\frac{{M\left(
{x_k - m } \right) }}{{ x_k\left( {M - m} \right)}}\right)f\left(
x_1 \right) + h\left( \frac{{m \left(M - x_k\right) }}{{ x_k\left(
{M - m}\right) }}\right)f\left( x_3 \right)}{f\left( x_1
\right)f\left( x_3 \right)}.
\end{align*}
Multiplying both sides by $h\left( {\frac{{w_k }}{{W_n }}}
\right)$ and summing up to $n$ we get
\begin{align*}
\sum\limits_{k = 1}^n{\frac{h\left( {\frac{{w_k }}{{W_n }}}
\right)}{f\left( x_k \right)} }&\ge \sum\limits_{k = 1}^n {\frac{
h\left(\frac{{M\left( {x_k - m } \right) }}{{ x_k\left( {M - m}
\right)}}\right)f\left(m \right) + h\left( \frac{{m \left(M -
x_k\right) }}{{ x_k\left( {M - m}\right) }}\right) f\left( M
\right)}{f\left( m \right)f\left( M \right) }h\left( {\frac{{w_k
}}{{W_n }}} \right)}.
\end{align*}
Reversing the order in the inequality again we get the required
result in \eqref{eq4.16}.
\end{enumerate}

\end{proof}

\begin{remark}
Theorem 22 and Corollary 23 in \cite{V}, can be extended to
 ${\rm{M_tN_h}}$-convexity in similar manner, we omit the
details.
\end{remark}
\begin{remark}
We note that, in this work, all results are valid  for
\begin{enumerate}
\item the class $\overline{\mathcal{MN}}\left(h,I\right)$,
whenever $h(t)=t$, $t\in[0,1]$

\item  the class $Q\left(I;{\rm{M_t}},{\rm{N_h}}\right)$, whenever
$h(t)=\frac{1}{t}$, $t\in(0,1)$

\item  the class $P\left(I;{\rm{M_t}},{\rm{N_h}}\right)$, whenever
$h(t)=1$, $t\in[0,1]$

\item  the class $K_s^2\left(I;{\rm{M_t}},{\rm{N_h}}\right)$,
whenever $h(t)=t^s$, $s\in (0,1]$ and $t\in[0,1]$.
\end{enumerate}
\end{remark}

\end{document}